\documentclass[12pt, a4paper]{amsart}
\usepackage[dvips]{epsfig}
\usepackage{amsmath}
\usepackage{amssymb}
\usepackage{amsfonts}
\usepackage{amsthm}
\usepackage{amsbsy}
\usepackage{amsgen}
\usepackage{amscd}
\usepackage{amsopn}
\usepackage{amstext}
\usepackage{amsxtra}

\usepackage{amssymb}
\usepackage{mathrsfs}
\usepackage{latexsym}
\usepackage[english]{babel}

\usepackage{graphicx}
\usepackage[usenames,dvipsnames]{xcolor}

\newtheorem{theorem}{Theorem}[section]
\newtheorem{remark}{Remark}
\newtheorem{proposition}[theorem]{Proposition}
\newtheorem{lemma}[theorem]{Lemma}
\newtheorem{corollary}[theorem]{Corollary}
\newtheorem{definition}{Definition}

\newtheorem{example}[theorem]{Example}
\newtheorem{notation}[theorem]{Notation}
\numberwithin{equation}{section}

\newcommand{\e}{{\rm e}}
\newcommand{\Cov}{{\rm Cov}}
\newcommand{\Var}{{\rm Var}}

\newcommand {\Z} {\mathbb{Z}}

\newcommand {\R} {\mathbb{R}}

\newcommand {\T} {\mathbb{T}}
\newcommand {\E} {\mathbb{E}}

\newcommand {\C} {\mathbb {C}}

\newcommand {\Cc} {\mathcal {C}}
\newcommand {\Zc} {\mathcal {Z}}
\newcommand{\Nc}{\mathcal N}
\newcommand {\Tb} {\mathbb {T}}
\newcommand{\Sc}{\mathcal S}

\newcommand{\supp}{\operatorname{supp}}

\newcommand{\new}[1]{{#1}}
\newcommand{\mnew}[1]{{#1}}

\def\paref#1{(\ref{#1})}
\newcommand{\tarc}{\mbox{\large$\frown$}}
\newcommand{\arc}[1]{\stackrel{\tarc}{#1}}

\title[Nodal intersections]{Asymptotic distribution of nodal intersections for arithmetic random waves}
\author{Maurizia Rossi and Igor Wigman}
\date{\today}

\begin{document}

\maketitle

\begin{abstract}

We study the nodal intersections number of random Gaussian toral Laplace eigenfunctions (``arithmetic random waves") against a fixed
smooth reference curve. The expected intersection number is proportional to the the square root of the eigenvalue times the length of curve, independent of its geometry. The asymptotic behaviour of the variance was addressed by Rudnick-Wigman; they found a precise
asymptotic law for ``generic" curves with nowhere vanishing curvature, depending on both its geometry and the angular distribution of
lattice points lying on circles corresponding to the Laplace eigenvalue. They also discovered that there exist peculiar ``static"
curves, with variance of smaller order of magnitude, though did not prescribe what the true asymptotic behaviour is in this case.

In this paper we study the finer aspects of the limit distribution of the nodal intersections number. For ``generic" curves we prove
the Central Limit Theorem (at least, for ``most" of the energies). For the aforementioned static curves we establish a
non-Gaussian limit theorem for the distribution of nodal intersections, and on the way find the true asymptotic behaviour
of their fluctuations, under
the well-separatedness assumption on the corresponding lattice points, satisfied by most of the eigenvalues.

\smallskip


\smallskip

\end{abstract}

\section{Introduction and main results}

\subsection{Toral nodal intersections}

Let $\mathbb T=\R^{2}/\Z^2$ be the two-dimensional standard torus and $\Delta$ the Laplacian on $\mathbb T$. It is well-known
that the eigenvalues of $-\Delta$ (``energy levels") are all the number of the form $E_{n}=4\pi^{2}n$ where $n$ is an integer expressible
as a sum of two squares $$n\in S:=\{ a^{2}+b^{2}:\: a,\, b\in  \Z\}.$$ Given a number $n\in S$ we denote
$$\Lambda_{n} = \{\lambda=(\lambda_1,\lambda_2)\in\Z^{2}:\: \|\lambda\|^{2}:=\lambda_1^2 + \lambda_2^2=n\}$$ to be the collection of lattice points lying on the radius-$\sqrt{n}$ centred
circle in $\R^{2}$\new{, and $\mathcal{N}_{n}=|\Lambda_{n}|$ to be their number};
the eigenspace of $-\Delta$ corresponding to $E_{n}$ then admits the orthonormal basis
$$\left \lbrace e_{\lambda}(x):=e^{i2\pi\langle\lambda,x\rangle} \right \rbrace_{\lambda\in\Lambda_{n}},$$ $x=(x_{1},x_{2})\in\mathbb T$.
Equivalently, we may express every (complex-valued) function $T_{n}$ satisfying
\begin{equation*}
\Delta T_{n}+E_{n}T_{n}=0
\end{equation*}
as a linear combination
\begin{equation}
\label{defrf}
T_n(x) := \frac{1}{\sqrt{\mathcal N_n}} \sum_{\lambda\in \Lambda_n} a_\lambda \e_\lambda(x)
\end{equation}
(the meaning of the normalizing constant on the r.h.s. of \eqref{defrf} will clear in \S \ref{sec:arith rand wav});
$T_{n}$ is real-valued if and only if for every $\lambda\in\Lambda_{n}$ we have
\begin{equation}
\label{eq:a_-lambda=conj}
a_{-\lambda}=\overline{a_{\lambda}}.
\end{equation}
From now on, we assume $T_n$ in \paref{defrf} to be real-valued.
The {\em nodal line} of $T_{n}$ is the zero set $T_{n}^{-1}(0)$; under some generic assumptions $T_{n}^{-1}(0)$ is a {\em smooth curve}.
Given a fixed reference curve $\Cc\subset \mathbb T$ one is interested in the number $\Zc_{\Cc}(T_{n})$
of {\em nodal intersection}, i.e. the number of
intersections of $T_{n}^{-1}(0)$ with $\Cc$; one expects that for $n$ sufficiently big $\Zc_{\Cc}(T_{n})$ is {\em finite},
and it is believed that their number should be commensurable with $\sqrt{n}$. That it is indeed so was verified by Bourgain and Rudnick
\cite{BR12} for $\Cc$ with nowhere vanishing curvature; they showed that in this case
\begin{equation}
\label{eq:sqrt{n}<<len<<sqrt{n}}
n^{1/2-o(1)} \ll \Zc_{\Cc}(T_{n}) \ll \sqrt{n},
\end{equation}
and got rid \cite{B-R15} of the $o(1)$ in the exponent of the lower bound for ``most" $n$.

\subsection{Nodal intersections for arithmetic random waves}
\label{sec:arith rand wav}

We endow the linear space \eqref{defrf} with a Gaussian probability measure by taking the coefficients $a_{\lambda}$ in \eqref{defrf}
random variables. Namely, we assume
that the $a_{\lambda}$ are standard complex-Gaussian i.i.d. save to \eqref{eq:a_-lambda=conj}, all defined on the same probability space;
equivalently, $T_{n}$ is a real-valued centered Gaussian field on $\mathbb T$ with covariance function
\begin{equation}\label{cov}
r_n(x,y) :=\E[T_{n}(x)\cdot T_{n}(y)]= \frac{1}{\mathcal N_n} \sum_{\lambda\in \Lambda_n} \cos(2\pi \langle \lambda, x-y\rangle).
\end{equation}
The random fields $T_{n}$ are the ``arithmetic random waves" \cite{O-R-W, K-K-W, M-P-R-W}; by \eqref{cov} above,
$T_n$ are unit variance and stationary, and by the standard abuse of notation we may denote $$r_n(x-y):=r_n(x,y). $$

Given a smooth (finite) length-$L$ curve $\Cc\subset \Tb^2$ we define $\Zc_{n}=\Zc_{\Cc}(T_n)$ to be the number of nodal intersections of
$T_{n}$ against $\Cc$; it is a (a.s. finite \cite{R-W}) random variable whose distribution is the primary focus of this paper.
Rudnick and Wigman \cite{R-W} have computed its expected number to be
\begin{equation}\label{igor_mean}
\E[\mathcal Z_n] = \frac{\sqrt{E_n}}{\pi \sqrt{2}}L
\end{equation}
independent of the geometry of $\Cc$, and also studied the asymptotic behaviour of the variance of $\Zc_{n}$ for large values of $n$;
in order to be able to exhibit their results we require some number theoretic preliminaries. First, denote $\Nc_{n} = |\Lambda_{n}|$
to be the number of lattice points lying on the radius-$\sqrt{n}$ circle. While on one hand, along $n\in S$ the number $\Nc_{n}$ grows \cite{Lan} {\em on average} as $c_{RL}\cdot \sqrt{\log n}$ with $c_{RL}>0$ the Ramanujan-Landau constant, on the other hand $\Nc_{n}$ is subject to large and erratic fluctuations; for example for the (thin) sequence of primes $p\equiv 1 \mod 4$ the corresponding $\mathcal N_{p} = 8$ does not grow at all.
From this point on we will assume that $$\new{\mathcal N_{n}\rightarrow\infty}$$ (also holding for \new{a} density-$1$ sequence $\{n\}\subset S$); it is also easy to derive \new{~\cite[p. 130]{MV}} the bound
\begin{equation}
\label{eq:N=O(n^o(1))}
\Nc_{n} = O(n^{o(1)}).
\end{equation}

We will also need to consider the angular distribution of $\Lambda_{n}$; to this end we define the probability measures
\begin{equation}\label{muenne}
\mu_n := \frac{1}{\Nc_{n}} \sum_{\lambda\in \Lambda_n} \delta_{\lambda/\sqrt n},
\end{equation}
$n\in S$, on the unit circle $\Sc^{1}\subset \R^{2}$.
It is well-known that for a density-$1$ subsequence $\{n_j\}\subset S$ the corresponding lattice points $\Lambda_{n_{j}}$ are asymptotically
equidistributed in the sense that the corresponding measures are weak-$*$ convergent
to the uniform measure on $\Sc^{1}$, \new{usually denoted
\begin{equation}
\label{eq:mujk=>dt/2pi equid}
\mu_{n_j} \Rightarrow \frac{d\theta}{2\pi}.
\end{equation}
}To the other extreme, there exists \cite{Ci} a (thin) sequence $\{n_{j} \}\subset S$ with
angles all concentrated around $\pm 1,\pm i$
\begin{equation}
\label{eq:munj conv Cil}
\mu_{n_j} \Rightarrow \tau_{0}:=\frac{1}{4}\left( \delta_{\pm 1}+\delta_{\pm i}  \right)
\end{equation}
(thinking of $\Sc^{1}\subset \C$), and the other partial weak-$*$ limits \new{of $\{\mu_{n}\}_{n\in S}$} were partially classified \cite{K-K-W, K-W}
\new{(``attainable measures")}.
\new{The measure $\tau_{0}$ (``Cilleruelo measure") together with its rotation by $\frac{\pi}{4}$ (``tilted Cilleruelo") are the two uniques
measures supported on $4$ points only, that are invariant w.r.t. rotation by $\frac{\pi}{2}$; both $\tau_{0}$ and its $\frac{\pi}{4}$-tilt
are weak-$*$ partial limits of $\{\mu_{n}\}_{n\in S}$.}

Back to the variance of $\Zc_{n}$, let $$\gamma=(\gamma_{1},\gamma_{2}) :[0,L]\rightarrow\mathbb T$$ be the arc-length parametrization of $\Cc$. Rudnick and Wigman
\cite[Theorem 1.1]{R-W} found that\footnote{Initially under another technical assumption, subsequently lifted in ~\cite{R-W-Y}.}
\begin{equation}
\label{igor_var}
\Var(\mathcal Z_n) = (4B_{\mathcal C}(\Lambda_n)-L^2)\cdot \frac{n}{\mathcal N_n} + O\left ( \frac{n}{\mathcal N_n^{3/2}}     \right ),
\end{equation}
where
\begin{equation}\label{B}
\begin{split}
B_{\mathcal C}(\Lambda_n) &:= \int_0^L \int_0^L \frac{1}{\mathcal N_n}
\sum_{\lambda\in \Lambda_n} \left \langle \frac{\lambda}{|\lambda|}, \dot \gamma(t_1) \right \rangle^2\cdot \left \langle \frac{\lambda}{|\lambda|}, \dot \gamma(t_2) \right \rangle^2\,dt_1 dt_2\\&=\int_0^L \int_0^L \int\limits_{\Sc^{1}}
\left \langle \theta, \dot \gamma(t_1) \right \rangle^2 \cdot \left \langle \theta, \dot \gamma(t_2) \right \rangle^2\,
d\mu_{n}(\theta)dt_{1}dt_{2}.
\end{split}
\end{equation}
The leading term $$4B_{\mathcal C}(\Lambda_n)-L^2$$ fluctuates \cite[\S7]{R-W} in the interval $[0,L^{2}]$ depending on both the angular distribution of the lattice points $\Lambda_{n}$ and the geometry of $\Cc$. One may also define $B_{\Cc}(\mu)$ for any probability measure $\mu$ on
$\Sc^{1}$, \new{invariant w.r.t. rotation by $\frac{\pi}{2}$,} with $\mu$ in place of $\mu_{n}$ on the r.h.s. of \eqref{B}.

\new{
In general, $B_{\Cc}(\mu)$ depends both on the geometry of $\gamma$, and $\mu$.
One might invert the integration order in \eqref{B} to write
\begin{equation}
\label{eq:B terms enrgy}
B_{\Cc}(\mu) = \int\limits_{\Sc^{1}}  E(\gamma;\theta)^{2}d\mu,
\end{equation}
where
\begin{equation*}
E(\gamma;\theta) = \int\limits_{\Cc}\left \langle \theta, \dot \gamma(t) \right \rangle^2dt,
\end{equation*}
is the ``$L^{2}$-energy of $\gamma$ along direction $\theta$", satisfying
\begin{equation}
\label{eq:enrgy constr}
E(\gamma;\theta)+E(\gamma;\theta^{\perp})=L
\end{equation}
for $\theta^{\perp}$ the rotation of $\theta$ by $\frac{\pi}{2}$.
Hence, under the constraint \eqref{eq:enrgy constr}, for every $\theta\in \Sc^{1}$ we have
\begin{equation}
\label{eq:1/2<=E^2+Eperp^2<=}
\frac{L^{2}}{2} \le E(\gamma;\theta)^{2}+E(\gamma;\theta^{\perp})^{2} \le L^{2}.
\end{equation}
Upon substituting \eqref{eq:1/2<=E^2+Eperp^2<=} into \eqref{eq:B terms enrgy}, and using the $\frac{\pi}{2}$-invariance of
$\mu$, we may deduce that
\begin{equation*}
\frac{L^{2}}{4} \le B_{\Cc}(\mu) \le \frac{L^{2}}{2},
\end{equation*}
so that the leading term in \eqref{igor_var} satisfies
\begin{equation}
\label{eq:lead const bnd}
0\le 4B_{\mathcal C}(\mu)-L^2 \le L^{2}.
\end{equation}
}

In order for the variance \eqref{igor_var} to observe an asymptotic law we need to split $S$ into sequences $\{n_{j}\}\subset S$ with corresponding $\Lambda_{n_{j}}$ admitting limit angular law. That is, $\mu_{n_{j}}\Rightarrow\mu$ for some $\mu$ probability measure on $\Sc^{1}$; in this case
$$B_{\Cc}(\Lambda_{n_{j}}) = B_{\Cc}(\mu_{n_{j}}) \rightarrow B_{\Cc}(\mu),$$ so that if $B_{\Cc}(\mu) > L^{2}/4$, in this case
\eqref{igor_var} is $$\Var(\mathcal Z_{n_j}) \sim (4B_{\mathcal C}(\mu)-L^2) \frac{n_j}{\mathcal N_{n_j}}.$$
Rudnick and Wigman \cite{R-W} observed that there exist ``special" curves (see Definition \ref{defstatic} in \S\ref{sec:main res}) for which
$$B_{\Cc}(\mu) \equiv \frac{L^{2}}{4}$$ \new{for all $\mu$ probability measures on $\Sc^{1}$ invariant w.r.t. rotation by $\frac{\pi}{2}$}
so that the leading term vanishes irrespective of the limit measure
\new{(though, given $\Cc$ the equality $$B_{\Cc}(\mu) = \frac{L^{2}}{4}$$ might still hold for some $\mu$ without holding identically, very restrictive in terms of $\mu$)}; for these \eqref{igor_var} gives no clue as of what is the true behaviour of $\Var(\Zc_{n})$ other than that it is of lower order of magnitude than the ``typical" $n/\Nc_{n}$.

\subsection{Statement of main results: limiting laws for $\Zc_{n}$}
\label{sec:main res}

Our two principal results below concern the limit laws for $\Zc_{n}$. Theorem \ref{mainth1} asserts the Central Limit Theorem for $\Zc_{n}$,
under the ``generic" scenario, \new{where} the leading terms of the variance in \eqref{igor_var} are bounded away from zero
(\new{cf. \S\ref{sec:discussion} for more details on how generic this situation is}). Indeed,
if we assume that the lattice points $\Lambda_{n}$ are equidistributed (i.e. the generic assumption
\eqref{eq:mujk=>dt/2pi equid} on the energy levels), then the assumptions of Theorem \ref{mainth1} hold for generic curves \cite[Corollary 7.2]{R-W}, also see the discussion in \S\ref{sec:discussion} below.
Theorem \ref{mainth2new}
investigates the peculiar alternative situation (the slightly more restrictive aforementioned ``static" curves)
with a non-Gaussian limit law, and among other things determines the true
(lower order) asymptotic law of the variance \eqref{igor_var} in the latter case.

\subsection*{Conventions} The notation $\mathop{\to}^{d}$ means convergence in distribution of random variables, $\mathop{=}^d $ denotes equality in law between
two random variables or random fields, and $\Nc(m,\sigma^2)$ is the Gaussian distribution with mean $m$ and variance $\sigma^2$. Throughout
this manuscript we will assume that $\Cc\subset \mathbb T$ is a given (fixed) curve of length $L$, and $\gamma:[0,L]\rightarrow\mathbb T$
an arc-length parametrization of $\Cc$.

\vspace{3mm}

First we formulate the Central Limit Theorem holding under ``generic" assumptions.

\begin{theorem}\label{mainth1}
Let $\mathcal C\subset \mathbb T$ be a smooth curve on the torus with nowhere zero
curvature of total length $L$, and $\lbrace n\rbrace \subset S$ such that $\mathcal N_n\to +\infty$, and $\lbrace 4B_{\mathcal C}(\mu_n) - L^2\rbrace$ is bounded away from zero. Then the limiting distribution of the nodal intersections number is Gaussian, i.e.
$$
\frac{\mathcal Z_n - \E[\mathcal Z_n]}{\sqrt{\Var(\mathcal Z_n)}}\mathop{\to}^{d} Z,
$$
where $Z\sim \mathcal N(0,1)$.
\end{theorem}

Next investigate the (non-generic) situation when the variance is of lower order. In order to formulate Theorem \ref{mainth2new} we will have to restrict $\Cc$ to be static (Definition \ref{defstatic}), and the sequence $\{n \}\subset S$ to be $\delta$-separated (Definition \ref{def:delta-sep}).

\begin{definition}[Static curves]\label{defstatic}
A smooth curve $\mathcal C \subset \mathbb T$ with nowhere zero curvature is called static if for every probability measure $\mu$ on $\mathcal S^1$
\begin{equation}\label{eqstatic}
4B_{\mathcal C}(\mu)-L^2=0.
\end{equation}
\end{definition}

For example, any semi-circle or circle are static \cite[\S7.2]{R-W}. In Appendix \ref{Sdmitri} we show that any smooth curve with
nowhere vanishing curvature and invariant under some nontrivial rotation of finite order, is static. We thank D. Panov for pointing this out.

\begin{definition}[$\delta$-separated sequences]
\label{def:delta-sep}
Let $\delta>0$. A sequence $\lbrace n \rbrace\subset S$ of energy levels is $\delta$-separated if
\begin{equation}\label{generic}
\min_{\lambda \ne \lambda'\in \Lambda_n} \|\lambda - \lambda'\| \gg n^{1/4 + \delta}.
\end{equation}
\end{definition}

Bourgain and Rudnick \cite[Lemma $5$]{B-R} showed that ``most" $n$ satisfy the $\delta$-separatedness property for every $0<\delta<\frac{1}{4}$.
In fact they have a strong quantitative estimate on the number of the exceptions; a precise estimate on these
was established more recently \cite{G-W}.

\vspace{3mm}

Now we introduce some more notation.

\begin{notation}
\begin{enumerate}
\item For $t\in [0,L]$ set
\begin{equation}\label{defF}
\begin{split}
f(t) &:=\dot \gamma_1(t)^2  -\frac{1}{L}\int_0^L  \dot \gamma_1(u)^2\,du = -\dot \gamma_2(t)^2  +\frac{1}{L}\int_0^L  \dot \gamma_2(u)^2\,du,\cr
g(t) &:=\dot \gamma_1(t)\dot \gamma_2(t)  -\frac{1}{L}\int_0^L  \dot \gamma_1(u)\dot \gamma_2(u)\,du.
\end{split}
\end{equation}

\item
For a probability measure $\mu$ on $\mathcal S^1$ define
\begin{equation}\label{Amu}
A_{\mathcal C}(\mu):=\int_{\mathcal S^1} \int_{\mathcal S^1}\left( \int_0^L  \langle \theta, \dot \gamma(t)  \rangle^2\cdot
\langle \theta', \dot \gamma(t)  \rangle^2\,dt\right)^2 \,d\mu(\theta)d\mu(\theta').
\end{equation}

\item
Also define the random variable
\begin{equation}\label{M}
\mathcal M(\mu) := \frac{1}{\sqrt{16 A_{\mathcal C}(\mu) - L^2}}\left(a_1(\mu) (Z_1^2-1)+ a_2(\mu) (Z_2^2-1) + a_3(\mu) Z_1 Z_2 \right),
\end{equation}
with $Z_1, Z_2$ i.i.d. standard Gaussian random variables and
\begin{equation}
\begin{split}
a_1(\mu) &:= 2(1+\widehat{\mu}(4) ) \int_0^L f(t)^2\,dt,\qquad
 a_2(\mu) := 2(1-\widehat{\mu}(4))\int_0^L f(t)^2\,dt,\cr
a_3(\mu) &:= 4 \sqrt{1-\widehat{\mu}(4)^2}\int_0^L f(t) g(t)\,dt,
\end{split}
\end{equation}
where
$$
\widehat \mu(4):=\int_{\Sc^1} z^{-4}\,d\mu(z)
$$ is the $4$th Fourier coefficient of $\mu$; \new{$\widehat \mu(4)$ is the first
nontrivial Fourier coefficient, as $\widehat \mu(0)=1$, and, by the $\frac{\pi}{2}$-invariance
of $\mu$, we have $\widehat \mu(k)=0$ unless
$k$ is divisible by $4$}.

\end{enumerate}

\end{notation}

\mnew{\new{Alternatively we can define the quantities in \paref{defF} in the following way. For} every $t\in [0,L]$
let $\phi(t)\in [0,2\pi]$ be the argument of $\dot{\gamma}(t)$, where
$\dot{\gamma}(t)$ is viewed as a (unit modulus) complex number, i.e. \begin{equation}\label{ang_vel}
\dot \gamma(t) =: \e^{i\phi(t)},
\end{equation}
then
\begin{equation*}
\begin{split}
f(t) &=\cos^2 \phi(t)  -\frac{1}{L}\int_0^L   \cos^2\phi(u)\,du = - \sin^2 \phi(t) +\frac{1}{L}\int_0^L  \sin^2 \phi(u)\,du,\cr
g(t) &=\cos \phi(t)\sin \phi(t)  -\frac{1}{L}\int_0^L \cos \phi(u)\sin \phi(u) \,du.
\end{split}
\end{equation*}
}

We are now in a position to formulate our second principal theorem.

\begin{theorem}\label{mainth2new}
Let $\mathcal C\subset \mathbb T$ be a static curve of length $L$, and
$\lbrace n\rbrace\subset S$ a $\delta$-separated sequence of energies such that $\mathcal N_n\to +\infty$.

\begin{enumerate}
\item The variance of $\Zc_{n}$ is asymptotic to
\begin{equation}\label{eqvar4}
\Var(\mathcal Z_n) =  \frac{n}{4\mathcal N^2_n} \left(16A_{\mathcal C}(\mu_n) -L^2\right )\cdot (1 +o(1)),
\end{equation}
with the leading term $16A_{\mathcal C}(\mu_n) -L^2$ bounded away from zero \new{under the sole assumption that
$\Cc$ is static}.

\item

There exists a coupling of the random variables
$\Zc_n$ and $\mathcal M(\mu_n)$ (defined in \paref{M}), such that
\begin{equation*}
\E\left[ \left |  \frac{\mathcal Z_n -\E[\mathcal Z_n]}{\sqrt{\Var(\mathcal Z_n)}} - \mathcal M(\mu_n)  \right |  \right]\to 0,
\end{equation*}
and
\begin{equation*}
 \frac{\mathcal Z_n -\E[\mathcal Z_n]}{\sqrt{\Var(\mathcal Z_n)}} - \mathcal M(\mu_n)  \to 0,\quad a.s.
\end{equation*}
\end{enumerate}
\end{theorem}

\new{The second part of Theorem \ref{mainth2new} establishes that the distribution law of
the normalised intersections number $$\frac{\mathcal Z_n -\E[\mathcal Z_n]}{\sqrt{\Var(\mathcal Z_n)}}$$
is asymptotic to that of  $\mathcal M(\mu_n)$, both in $L^{1}$-sense and a.s.
The quantity $A_{\mathcal C}(\mu_n)$ (and also $16A_{\mathcal C}(\mu_n) -L^2$) depend on both the geometry of $\Cc$
and $\mu_{n}$ in a nontrivial way.}
In the particular case of $\mathcal C$ a full circle, and $\mu_n \Rightarrow d\theta/2\pi$, Theorem \ref{mainth2new} yields the following
via a routine computation:

\begin{example}
Let $\mathcal C\subset \mathbb T$ be a full circle of total length $L$, let $\lbrace n\rbrace\subset S$ be a $\delta$-separated sequence such that $\mathcal N_n\to +\infty$ and $\mu_n \Rightarrow \frac{d\theta}{2\pi}$, then for the variance of the nodal intersections number we have
$$
\Var \left( \mathcal Z_n \right)\sim  \frac{L^2}{32} \cdot \frac{n}{\mathcal N^2_n},
$$
and the limiting distribution is
\begin{equation}
\label{eq:expl circ distr}
\frac{\mathcal Z_n - \E[\mathcal Z_n]}{\sqrt{\Var(\mathcal Z_n)}}\mathop{\to}^{d} 1-\frac{Z_1^2 + Z_2^2}{2},
\end{equation}
where $Z_1,Z_2$ are i.i.d. standard Gaussian random variables. \new{Note that the r.h.s. of \eqref{eq:expl circ distr}
is bounded above by $1$;
to our best knowledge the l.h.s. of \eqref{eq:expl circ distr} might be bigger than $1$, so, in light of the above, bounded above by $1$ with high probability as $n\rightarrow\infty$.}
\end{example}

\subsection{Discussion}
\label{sec:discussion}

Given a length-$L$ toral curve $\Cc$ and its arc-length parametrization $$\gamma:[0,L]\rightarrow\Tb$$ we may ~\cite[\S 7.3]{R-W} associate a complex number $I(\gamma)\in \C$ in the following way. \mnew{Let us keep in mind \paref{ang_vel}: for every $t\in [0,L]$ let $\phi(t)\in [0,2\pi]$ be the argument of $\dot{\gamma}(t)$, where
$\dot{\gamma}(t)$ is viewed as a (unit modulus) complex number, i.e.
$\dot{\gamma}(t)= e^{i\phi(t)}$}, we then set
\begin{equation*}
I(\gamma) := \int\limits_{0}^{L}e^{2i\phi(t)}dt.
\end{equation*}
\new{\cite[Corollary 7.2]{R-W} asserts that:

\begin{enumerate}

\item The curve $\Cc$ is static if and only if $I(\gamma)=0$.

\item If $I(\gamma)\ne 0$ and $\Re(I(\gamma)) =  0$ (resp. $\Im(I(\gamma)) =  0$), then for all $\mu$ measures on $\Sc^{1}$
invariant w.r.t. rotation by $\frac{\pi}{2}$ and complex conjugation,
$4B_{\mathcal C}(\mu)-L^2=0$ if and only if
$\mu=\tau_{0}$
is the Cilleruelo measure on the r.h.s. of \eqref{eq:munj conv Cil} (resp. $4B_{\mathcal C}(\mu)-L^2=0$ if and only if
$\mu$ is the $\frac{\pi}{4}$-tilt of $\tau_{0}$ on the r.h.s. of \eqref{eq:munj conv Cil}).

\end{enumerate}

}

\new{It follows from above that if} both the imaginary and real parts
\begin{equation}
\label{eq:Re,Im(I) ne 0}
\Im (I(\gamma)), \Re (I(\gamma)) \ne 0
\end{equation}
of $I(\gamma)$
do not vanish, then ~\cite[Corollary 7.2]{R-W} show that $4B_{\Cc}(\mu)-L^{2}$ is bounded away from zero for {\em all} probability measures $\mu$ on $\Sc^{1}$ invariant w.r.t. rotation by $\frac{\pi}{2}$ and complex conjugation. Hence in this case \new{\eqref{eq:Re,Im(I) ne 0},}
the assumptions of Theorem \ref{mainth1} are satisfied for the full sequence $n\in S$ of energy levels
\new{(still assuming $\mathcal N_n\to +\infty$)}. The condition \eqref{eq:Re,Im(I) ne 0} is a generic condition on $\Cc$\new{,} understood, for example, in the sense of {\em prevalence}, see\footnote{We wish to thank Michael Benedicks for pointing out ~\cite{O-Y} to us.}
e.g. ~\cite[\S 6]{O-Y} (\new{in particular} Example $3.6$) and references therein.
\new{Hence the scenario described by Theorem \ref{mainth1} is ``generic":
it holds for  both ``generic" curves for all energy levels, and also the non-generic curves with $I(\gamma) \in \R,i\R$ either real or purely imaginary, for generic energy levels. Note that the condition \ref{eq:Re,Im(I) ne 0} is {\em not} invariant under rotations of $\gamma$,
so neither is the distribution of $\Zc_{n}$, nor its variance.}

The two only cases not covered by theorems \ref{mainth1} and \ref{mainth2new} are then:
\begin{enumerate}

\item We have $\Re (I(\gamma)) =0 $, $\Im (I(\gamma))\ne 0$ and the lattice points corresponding to the subsequence $\{n\}\subseteq S$ converge to
$$\mu_{n}\Rightarrow \frac{1}{4}\left(\delta_{\pm 1}+\delta_{\pm i}\right)$$ the Cilleruelo measure.

\item We have $\Re (I(\gamma)) \ne 0 $, $\Im (I(\gamma))= 0$ and the lattice points corresponding to the subsequence $\{n\}\subseteq S$ converge to $$\mu_{n}\Rightarrow \frac{1}{4}\delta_{\pi/4} \star \left(\delta_{\pm 1}+\delta_{\pm i}\right)$$ the tilted Cilleruelo measure,
i.e. Cilleruelo measure rotated by $\frac{\pi}{4}$.
\end{enumerate}
In order to analyse either of \new{the scenarios (1) and (2)} one needs to understand which of the terms
$$\frac{n}{\Nc_{n}}\cdot (4B_{\Cc}-L^{2})$$ (leading term of the $2$nd chaotic projection, see \eqref{varcomp})
or $\frac{n}{\Nc_{n}^{2}}$ (the order of magnitude of the next term) is dominant by order
of magnitude, knowing that in \new{either case (1) or (2), leading coefficient}
$4B_{\Cc}-L^{2}$ vanishes asymptotically, at least under the $\delta$-separatedness
assumption. Equivalently, whether
$4B_{\Cc}-L^{2}$ vanishes more rapidly than $\frac{1}{\Nc_{n}}$; this would most certainly involve the rate of
convergence of $\mu_{n}$, and it is plausible that one can construct sequences $\{n\}$
observing both kinds of behaviour.

\new{Towards the end of the introduction we would like to discuss the possibility of extending the presented results
to higher dimensional tori (that is, the distribution of the nodal intersections number of higher dimensional arithmetic random waves
against a smooth curve). In this case the associated lattice points problems are significantly more delicate, and we
are not aware of precise asymptotic results concerning the variance, let alone the more subtle question of the limit
law, with only a few partial results published \cite{R-W-Y,Ma}, for non-vanishing curvature curves and straight segments
respectively.} \mnew{Finally \new{we use this opportunity \mnew{to} point out the recent work \cite{VanVu} on the
universality of various statistics related to the nodal intersections number
(\new{such} as its mean and \new{higher moments}) w.r.t. the law of random coefficients $a_\lambda$ in \paref{defrf}}.}

\subsection*{Acknowledgements}

We express our deep gratitude to Ze\'{e}v Rudnick for many stimulating and fruitful discussions, and, in particular, pointing out
\cite[Lemma $5$]{B-R}, simplifying the proof of Lemma \ref{lem:angle chords}, and also his comments on an earlier version of this
manuscript. We are indebted to Dmitri Panov for freely sharing his expertise in geometry, who, among other things, has constructed a family of static curves in Proposition \ref{prop:Dima invar 2pi/k} and kindly allowed us to include it in the paper.
It is a pleasure to thank Michael Benedicks, P\"{a}r Kurlberg and Domenico Marinucci for
their comments on an earlier version of this manuscript.
M.R. would like to heartily thank the Department of Mathematics at King's College London for its warm hospitality.
\new{Finally, we are grateful to the anonymous referees and the editor for helping us improve the readability of our manuscript}.

The research leading to these results has received funding from the
European Research Council under the European Union's Seventh
Framework Programme (FP7/2007-2013), ERC grant agreement
n$^{\text{o}}$ 335141 (I.W.), and the grant F1R-MTH-PUL-15STAR (STARS) at University of Luxembourg (M.R.). \mnew{The research of M.R. is currently supported by the Fondation Sciences Math\'ematiques de Paris and the ANR-17-CE40-0008 project \emph{Unirandom}.}

\section{Outline of the paper}

\subsection{On the proof of the main results}

The proofs of theorems \ref{mainth1} and \ref{mainth2new} are based on the \emph{chaotic expansion} for the nodal intersections number (see \S\ref{Schaos} \mnew{and Appendix \ref{appendix_chaos}}).
We first consider a unit speed parametrization of the curve $\gamma: [0,L] \to \mathcal C$, and set
\begin{equation}\label{process}
f_n : [0,L] \to \mathbb R;\qquad t\mapsto T_n(\gamma(t)).
\end{equation}
The map $f_n$ in \paref{process} defines a (non-stationary) centered Gaussian process on $[0,L]$ with covariance function
\begin{equation}\label{covf}
r_n(t_1, t_2):=\Cov(f_n(t_1), f_n(t_2)) = \frac{1}{\mathcal N_n} \sum_{\lambda\in \Lambda_n} \cos(2\pi \langle \lambda, \gamma(t_1) - \gamma(t_2)\rangle),\quad t_1,t_2\in [0,L],
\end{equation}
see \eqref{cov}.
The number $\mathcal Z_n$ of nodal intersections of $T_{n}$ against $\Cc$ equals to the number of zero crossings
of $f_n$ in the interval $[0,L]$, and we can formally write
\begin{equation}\label{formal}
\mathcal Z_n = \int_0^L \delta_0(f_n(t))|f'_n(t)|\,dt,
\end{equation}
where $\delta_0$ is the Dirac delta function.

The random variable $\mathcal Z_n$ in \paref{formal} admits the Wiener-\^Ito chaotic expansion (see \S\ref{Schaos} \mnew{and Appendix \ref{appendix_chaos}}) of the form
\begin{equation}\label{chaos_exp}
\mathcal Z_n = \sum_{q=0} \mathcal Z_n[2q],
\end{equation}
where the above series converges in the space $L^2({\mathbb P})$ of random variables with finite variance. In particular, the random variables $\mathcal Z_n[2q]$, $\mathcal Z_n[2q']$ are orthogonal (uncorrelated) for $q\ne q'$, and $\mathcal Z_n[0]=\E[\mathcal Z_n]$.

Evaluating the second chaotic projection $\mathcal Z_n[2]$ yields that, under the assumptions in Theorem \ref{mainth1}, the variance of the total number $\mathcal Z_n$ of nodal intersections is asymptotic to the variance of $\mathcal Z_n[2]$, both being asymptotic to \paref{igor_var}. The latter and the orthogonality of the Wiener chaoses imply that the distribution of $\mathcal Z_n[2]$ dominates the series on the r.h.s. of \paref{chaos_exp}, and a Central Limit Theorem result for $\mathcal Z_n[2]$ allows to infer the statement of Theorem \ref{mainth1}.

\vspace{3mm}

Assume now that the curve $\mathcal C$ is static (Definition \ref{defstatic}). The leading term in \paref{igor_var} vanishes,
and we are left only with an upper bound for the variance of $\mathcal Z_n$.
To obtain its precise asymptotics we need to inspect the proof of the \emph{approximate} Kac-Rice formula
\cite[Proposition 1.3]{R-W}, and obtain one \new{extra} term in the \new{Taylor} expansion \new{of the two-point correlation function}
(see \new{\eqref{kr-approx} in} \S\ref{SproofKac}). The main difficulty is how to control ``off-diagonal" terms coming from the fourth moment of
\new{the covariance function $r_{n}$ in \eqref{cov}}, and specifically the quantity
\begin{equation}\label{quantity4}
\frac{1}{\mathcal N_n^2}\sum_{\substack{\lambda_1, \lambda_2,\lambda_3,\lambda_4\in \Lambda_n \\ \lambda_{1}+\lambda_{2}+\lambda_{3}+\lambda_{4}\ne 0}} \frac{1}{\| \lambda_1 + \lambda_2 + \lambda_3 + \lambda_4 \|},
\end{equation}
\new{required to be evaluated in order to control the fourth moment of $r_{n}$}.
We will use some properties of the $\delta$-separated sequences of energy levels (Definition \ref{def:delta-sep}) to show that
\new{the expression in} \paref{quantity4} is $o(1)$ (see Lemma \ref{remainder4}), and then prove \paref{eqvar4}.

It turns out that the leading term in the chaotic expansion \paref{chaos_exp} is no longer the projection onto the second chaos, but the projection $\mathcal Z_n[4]$ onto the fourth chaos. A precise analysis of the latter allows to get its asymptotic (non-Gaussian) distribution in \paref{M}, thus concluding the proof of Theorem \ref{mainth2new} by a standard application of \cite[Theorem 11.7.1]{Du}.

\subsection{Chaos expansion}\label{Schaos}

In this section we compute the chaotic expansion \paref{chaos_exp} for the nodal intersections number $\mathcal Z_n$.
The reader can refer to \new{Appendix \ref{appendix_chaos}, where the discussion is focussed on our particular situation}, or \cite{N-P},
for a complete discussion on Wiener-\^Ito chaos expansions.
Recall the definition \paref{process} of the random process $f_n$ and the formal expression \paref{formal}.
Note that for every $t\in [0,L]$
$$
f'_n(t) = \langle \nabla T_n (\gamma(t)), \dot \gamma(t)\rangle,
$$
where $\nabla T_n$ denotes the gradient of $T_n$ and $\dot \gamma$ the first derivative of $\gamma$. We have \cite[Lemma 1.1]{R-W} that
\begin{equation}\label{alpha_var}
\Var(f'_n(t))= 2\pi^2n =:\alpha.
\end{equation}
We can then rewrite \paref{formal} as
\begin{equation}\label{formal2}
\mathcal Z_n = \sqrt{2\pi^2n} \int_0^L \delta_0 (f_n(t)) |\widetilde f_n'(t)|\,dt,
\end{equation}
where
$$
\widetilde f_n'(t) := \frac{f_n'(t)}{\sqrt{2\pi^2n}},\quad t\in [0,L].
$$

As $f_{n}$ is a unit variance process, \new{differentiating the equality $r(t,t)\equiv 1$ yields that} for every $t\in [0,L]$, $f_n(t)$ and $f'_n(t)$ are independent (see e.g. \cite[Lemma 2.2]{R-W}),
and so are $f_n(t)$ and $\widetilde f'_n(t)$. \mnew{Heuristically, in order to find the chaotic expansion for $\mathcal Z_n$ in \paref{formal2} we first compute the chaotic decompositions\footnote{Note that the chaotic decomposition of $\delta_0(Z)$ is just a \emph{formal} expansion} for $\delta_0(Z)$ and $|W|$, where $(Z,W)$ is a standard Gaussian random vector. Since both the Dirac mass at $0$ and the absolute value are even functions, in these expansions only terms corresponding to even order Wiener chaoses appear:
\begin{equation}\label{decompositions}
\delta_0(Z) = \sum_{q=0}^{+\infty} b_{2q} H_{2q}(Z),\qquad |W|= \sum_{\ell=0}^{+\infty} a_{2\ell} H_{2\ell}(W),
\end{equation}
where $\lbrace H_k, k=0,1,\dots \rbrace$ denotes the Hermite polynomials (see \cite[\S5.5]{Sz75} and Appendix \ref{appendix_chaos})
and the coefficients $b_{2k}, a_{2k}, k=0, 1, 2, \dots$ are given by (see e.g. \cite{K-L}):
\begin{equation}\label{b}
b_{2q} := \frac{1}{(2q)! \sqrt{2\pi}}H_{2q}(0),\quad q\ge 0,
\end{equation}
and
\begin{equation}\label{a}
a_{2\ell}:= \sqrt{\frac{2}{\pi}}\frac{(-1)^{\ell+1}}{2^\ell \ell! (2\ell-1)},\quad \ell\ge 0.
\end{equation}
In order to obtain the (formal) chaotic expansion of $\delta_0(Z) |W|$, the random variables $Z$ and $W$ being independent, it suffices to ``multiply" the two series in \paref{decompositions}: the $2q$-th chaotic component is given by the sum of all terms of the form $b_{2m}a_{2n}H_{2m}(Z)H_{2n}(W)$ for $m+n = q$ (see the so-called product formula \cite[Theorem 2.7.10]{N-P}). To obtain the chaotic expansion for $\mathcal Z_n$ in \paref{formal2}, \new{on heuristic level}, we replace $(Z,W)$ by $(f_n(t), \widetilde{f'_n}(t))$ and then integrate over $t\in [0,L]$ the (formal) series expansion for $\delta_0(f_n(t))|\widetilde{f_n'}(t)|$.
}

\mnew{To be more precise, from} \cite[Lemma 2]{K-L} we have the chaotic expansion \paref{chaos_exp} for \mnew{the nodal intersection number}  \paref{formal2}
\begin{equation}\label{chaosexp}
\mathcal Z_n = \sum_{q=0}^{+\infty} \mathcal Z_n[2q]= \sqrt{2\pi^2n} \sum_{q=0}^{+\infty} \sum_{\ell=0}^{q}
b_{2q-2\ell} a_{2\ell} \int_0^L H_{2q-2\ell}(f_n(t)) H_{2\ell}(\widetilde f_n'(t))\,dt;
\end{equation}
in particular from \paref{chaosexp} we have for $q\ge 0$
\begin{equation}\label{proj}
\mathcal Z_n[2q] = \sqrt{2\pi^2n}  \sum_{\ell=0}^{q}
b_{2q-2\ell} a_{2\ell} \int_0^L H_{2q-2\ell}(f_n(t)) H_{2\ell}(\widetilde f_n'(t))\,dt.
\end{equation}

\subsection{Plan of the paper}

In \S\ref{Sproof} we will state a few key propositions instrumental in proving theorems \ref{mainth1} and \ref{mainth2new}, in particular, concerning $\Zc_{n}[2]$ in \paref{proj} for ``generic" curves (\new{for} Theorem \ref{mainth1}),
and $\Zc_{n}[4]$, and also the \new{``approximate Kac-Rice formula"
\eqref{kr-approx}} for static curves (\new{for} Theorem \ref{mainth2new}).
In \S\ref{Sproofs2} we will then investigate the second chaotic component for ``generic" curves, whereas in \S \ref{Schaos4} the fourth one in the case of static curves. The proof of the approximate Kac-Rice formula for the variance of the number of nodal intersections against static curves will be given in \S\ref{SproofKac}.

In the Appendix we will collect some \mnew{background} and \mnew{several} technical results (concerning chaotic components, approximate Kac-Rice formula and bounds for certain summations over lattice points such as \paref{quantity4}) and, in particular, in \S\ref{Sdmitri} a family of static curves will be constructed.

\section{Proofs of the main results}\label{Sproof}

The zeroth term in the chaos expansion \paref{chaosexp} of $\Zc_{n}$ \mnew{is the orthogonal projection of $\Zc_{n}$ onto the zeroth Wiener chaos (i.e., onto $\R$), hence it coincides with the mean of the random variable itself, consistently with the following lemma (cf. \eqref{igor_mean}).}

\begin{lemma}\label{lemma0}
For every $n\in S$,
\begin{equation}\label{proj0}
\mathcal Z_n[0] = \frac{\sqrt{E_n}}{\pi \sqrt{2}}L.
\end{equation}
\end{lemma}
\begin{proof}
From \paref{b}, \paref{a} and \paref{proj} we have, for $q=0$,
$$
\mathcal Z_n[0] = \sqrt{2\pi^2n}\, b_0\, a_0\,L=  \sqrt{2\pi^2n}\, \frac{1}{\sqrt{2\pi}} \sqrt{\frac{2}{\pi}}\,L = \sqrt{2n} L = \frac{\sqrt{E_n}}{\pi\sqrt{2}} L.
$$
\end{proof}

\subsection{Proof of Theorem \ref{mainth1}}

To prove Theorem \ref{mainth1}, we first need to study the asymptotic variance of the second chaotic component $\Zc_{n}[2]$. \new{It turns
our that, under the scenario covered by Theorem \ref{mainth1}, the second chaotic component $\Zc_{n}[2]$ dominates the distribution of
$\Zc_{n}$ (a by-product of Proposition \ref{secondvar} below and the orthogonality of the Wiener chaos spaces), and the convergence in
distribution of $\Zc_{n}[2]$ to Gaussian (Proposition \ref{prop1}) will allow us to deduce the same for $\Zc_{n}$.}

\begin{proposition}\label{secondvar}
Let $\mathcal C\subset \mathbb T$ be a smooth curve on the torus with nowhere zero
curvature, of total length L, and $\lbrace n\rbrace \subset S$  such that  $\mathcal N_n\to +\infty$, and $\lbrace 4B_{\mathcal C}(\mu_n) - L^2\rbrace$ is bounded away from zero. Then
$$
\Var(\mathcal Z_n)\sim \Var(\mathcal Z_n[2]),
$$
i.e. the variance of $\Zc_{n}[2]$ is asymptotic to the variance \paref{igor_var} of the nodal intersections $\Zc_{n}$.
\end{proposition}

In light of Proposition \ref{secondvar} to be proven in \S\ref{Sproofs2}, we are to study the asymptotic distribution of the second chaotic component.

\begin{proposition}\label{prop1}
Let $\mathcal C\subset \mathbb T$ be a smooth curve on the torus with nowhere zero
curvature, of total length L, and $\lbrace n\rbrace \subset S$ such that $\mathcal N_n\to +\infty$, and $\lbrace 4B_{\mathcal C}(\mu_n) - L^2\rbrace$ is bounded away from zero. Then
$$\frac{\mathcal Z_n[2]}{\sqrt{\Var(\mathcal Z_n[2])}}\mathop{\to}^{d} Z,
$$
where $Z\sim \mathcal N(0,1)$.
\end{proposition}

Proposition \ref{prop1} will be proven in \S\ref{Sproofs2}. We are in a position to prove our first main result.

\begin{proof}[Proof of Theorem \ref{mainth1} assuming propositions \ref{secondvar} and \ref{prop1}]
By \paref{chaosexp} we write
$$
\frac{\mathcal Z_n - \E[\mathcal Z_n]}{\sqrt{\Var(\mathcal Z_n)}}= \sum_{q=1}^{+\infty} \frac{\mathcal Z_n[2q]}{\sqrt{\Var(\mathcal Z_n)}}.
$$
Thanks to Proposition \ref{secondvar} and the orthogonality of Wiener chaoses we have, as $\mathcal N_n\to \infty$,
$$
\frac{\mathcal Z_n - \E[\mathcal Z_n]}{\sqrt{\Var(\mathcal Z_n)}} =  \frac{\mathcal Z_n[2]}{\sqrt{\Var(\mathcal Z_n[2])}}  + o_{\mathbb P}(1),
$$
where $o_{\mathbb P}(1)$ denotes a sequence of random variables converging to zero in probability. In particular, the distribution
of the normalized total number of nodal intersections is asymptotic to the distribution of $\frac{\mathcal Z_n[2]}{\sqrt{\Var(\mathcal Z_n[2])}}$.
The latter and Proposition \ref{prop1} imply the statement of Theorem \ref{mainth1}.

\end{proof}

\subsection{Proof of Theorem \ref{mainth2new}}

\new{It turns
our that, under the scenario covered by Theorem \ref{mainth2new}, the $4$th chaotic component $\Zc_{n}[4]$ dominates the distribution of
$\Zc_{n}$ (a by-product of propositions \ref{kac-rice} and \ref{prop-var4} below, and the orthogonality of the Wiener chaos spaces).
The distribution law of $\Zc_{n}$ is then asymptotic to the one of $\Zc_{n}[4]$ in Proposition \ref{4chaos}.
Though in this scenario, the $2$nd chaotic component $\Zc_{n}[2]$ does not vanish precisely, it does not
contribute to the asymptotic law of $\Zc_{n}$ (also follows from Proposition \ref{kac-rice}). Unlike propositions
\ref{secondvar}-\ref{prop1}, to prove the more delicate propositions
\ref{kac-rice}-\ref{prop-var4} we will invoke both the assumptions on $\Cc$ and $\{n\}$, i.e. $\Cc$ is static and $\{n\}$ is $\delta$-separated.}
Our first proposition asserts the variance part \eqref{eqvar4} of Theorem \ref{mainth2new}, to be proven in \S \ref{SproofKac}.

\begin{proposition}\label{kac-rice}
Let $\mathcal C\subset \mathbb T$ be a static curve of length $L$, and $\lbrace n\rbrace\subset S$ be a $\delta$-separated sequence such that $\mathcal N_n\to +\infty$. Then
\begin{equation}\label{eqvar42}
\Var(\mathcal Z_n) =  \frac{n}{4\mathcal N^2_n} \left(16A_{\mathcal C}(\mu_n) -L^2\right ) (1 +o(1)),
\end{equation}
where $A_{\mathcal C}(\mu_n)$ is given in \paref{Amu} with $\mu=\mu_n$.
Moreover, the leading term $16A_{\mathcal C}(\mu_n) -L^2$ in \paref{eqvar42} is bounded away from zero.
\end{proposition}

Next we assert that, under the assumptions of Theorem \ref{mainth2new}, the fourth term in the chaotic series \paref{chaosexp} dominates.

\begin{proposition}\label{prop-var4}
Let $\mathcal C\subset \mathbb T$ be a static curve of length $L$, and $\lbrace n\rbrace\subset S$ a $\delta$-separated sequence
such that $\mathcal N_n\to +\infty$. Then
\begin{equation}\label{var_chaos2}
\Var(\mathcal Z_n[2]) = o\left(\frac{n}{\mathcal N_n^2}\right),
\end{equation}
and
\begin{equation}\label{var_chaos4}
\Var(\mathcal Z_n[4]) \sim \frac{n}{4\mathcal N^2_n} \left(16A_{\mathcal C}(\mu_n) -L^2\right ).
\end{equation}
\end{proposition}

The above implies that it suffices to study the asymptotic distribution of the fourth chaotic projection.

\begin{proposition}\label{4chaos}
Let $\mathcal C\subset \mathbb T$ be a static curve of length $L$, and $\lbrace n\rbrace\subset S$ a $\delta$-separated sequence such that
$\mathcal N_n\to +\infty$, and $\mu_n \Rightarrow \mu$. Then
\begin{equation}
\label{eq:Zn4->M(mu)}
\frac{\mathcal Z_n[4] }{\sqrt{\Var(\mathcal Z_n[4])}}  \mathop{\to}^{d}  \mathcal M(\mu),
\end{equation}
where $\mathcal M(\mu)$ is given by \paref{M}.
\end{proposition}

\begin{proof}[Proof of Theorem \ref{mainth2new} assuming propositions \ref{kac-rice}-\ref{4chaos}]
By \paref{chaosexp} we write
$$
\frac{\mathcal Z_n - \E[\mathcal Z_n]}{\sqrt{\Var(\mathcal Z_n)}}= \sum_{q=1}^{+\infty} \frac{\mathcal Z_n[2q]}{\sqrt{\Var(\mathcal Z_n)}}.
$$
Thanks to Proposition \ref{kac-rice}, Proposition \ref{prop-var4} and the orthogonality of Wiener chaoses we have, as $\mathcal N_n\to \infty$,
$$
\frac{\mathcal Z_n - \E[\mathcal Z_n]}{\sqrt{\Var(\mathcal Z_n)}} =  \frac{\mathcal Z_n[4]}{\sqrt{\Var(\mathcal Z_n[4])}}  + o_{\mathbb P}(1),
$$
where $o_{\mathbb P}(1)$ denotes convergence to zero in probability. In particular, the distribution of the normalized total number of nodal intersections is asymptotic to the one of $$\frac{\mathcal Z_n[4]}{\sqrt{\Var(\mathcal Z_n[4])}}.$$ Therefore \eqref{eq:Zn4->M(mu)}
also holds with $\Zc_{n}$ in place of $\mathcal Z_n[4]$.
In particular, this implies
$$
d \left ( \frac{\mathcal Z_n - \E[\mathcal Z_n]}{\sqrt{\Var(\mathcal Z_n)}}, \mathcal M(\mu_n) \right )\to 0,
$$
where $d$ is any metric which metrizes convergence in distribution of random variables, or the Kolmogorov distance (see e.g. \cite[\S C]{N-P}), since $\mathcal M (\mu)$ in \paref{M}
has absolutely continuous distribution for arbitrary probability measure $\mu$. Theorem \ref{mainth2new} is then a direct consequence \new{of \cite[Theorem 11.7.1]{Du}} and of the fact that the sequence $\left \lbrace \left (\mathcal Z_n - \E[\mathcal Z_n]\right )/ \sqrt{\Var(\mathcal Z_n)} \right \rbrace$ is bounded in $L^2(\mathbb P)$.

\end{proof}

\section{Proofs of Proposition \ref{secondvar} and Proposition \ref{prop1}}\label{Sproofs2}

In this section we investigate the asymptotic behaviour of the second chaotic
component $\mathcal Z_n[2]$ of $\Zc_{n}$. Our starting point is the
decomposition of $\Zc_{n}[2]$ into ``diagonal" and ``off-diagonal" terms: we define
the diagonal term
\begin{equation}\label{eq2}
\begin{split}
\mathcal Z_n^a[2]:= \frac{\sqrt{2\pi^2n}}{2\pi}  \frac{1}{\mathcal N_n}\,2
\sum_{\lambda\in \Lambda_n^+}
(|a_\lambda|^2 -1) \left ( 2 \int_0^L
\left \langle \frac{\lambda}{|\lambda|}, \dot \gamma(t)\right \rangle^2\,dt  -L \right),
\end{split}
\end{equation}
where if $\sqrt{n}$ is not an integer
$
\Lambda_n^+ := \lbrace \lambda \in \Lambda_n : \lambda_2 >0 \rbrace,
$
otherwise
$$
\Lambda_n^+:=\lbrace \lambda\in \Lambda_n : \lambda_2 >0\rbrace \cup \lbrace (\sqrt{n}, 0)\rbrace.
$$
The off-diagonal term is
\begin{equation}\label{eq3}
\begin{split}
\mathcal Z_n^b[2]:= \frac{\sqrt{2\pi^2n}}{2\pi} \frac{1}{\mathcal N_n} \sum_{\lambda\ne \lambda'}
a_\lambda \overline{a_{\lambda'}} \int_0^L \left(2\left \langle \frac{\lambda}{|\lambda|},
\dot \gamma(t)\right \rangle \left \langle \frac{\lambda'}{|\lambda|}, \dot \gamma(t)\right \rangle
 -1 \right)\e^{i2\pi\langle \lambda -\lambda', \gamma(t)\rangle}\,dt.
\end{split}
\end{equation}
\new{Recall that, by construction, $Z_n[2]$, $Z_n^a[2]$ and $Z_n^b[2]$ are all mean zero, so their variances are
equal to their respective second moments.}

\begin{lemma}\label{lemproj2}
For every $n\in S$ we have
\begin{equation}\label{proj2formula}
\mathcal Z_n[2]=\mathcal Z_n^a[2]+\mathcal Z_n^b[2].
\end{equation}
\end{lemma}

The proof of Lemma \ref{lemproj2} will be given in Appendix \ref{Schaos2}. \new{We are now in a position to give
a proof for Proposition \ref{secondvar}; first it is useful to recall the following identity:
\begin{lemma}[{\cite[Lemma 2.3]{R-W2}}]
\label{lem:ident sum squares prod}
For every $z\in \R^{2}$ we have the following equality:
\begin{equation}
\label{eq:ident sum squares prod}
\frac{1}{\Nc_{n}}\sum_{\lambda\in \Lambda_n}\left\langle z ,\, \frac{\lambda}{|\lambda|}\right\rangle = \frac{1}{2}|z|^{2}.
\end{equation}
\end{lemma}
}

\begin{proof}[Proof of Proposition \ref{secondvar} assuming Lemma \ref{lemproj2}]
Lemma \ref{lemproj2} yields
\begin{equation}\label{vv}
\Var(\mathcal Z_n[2])=\Var(\mathcal Z_n^a[2])+\Var(\mathcal Z_n^b[2]) + 2 \Cov(\mathcal Z_n^a[2], \mathcal Z_n^b[2]).
\end{equation}
Let us first study $\Var(\mathcal Z_n^a[2])$. By the definition \eqref{eq2} of $\mathcal Z_n^a[2]$ we may compute its variance to be
\begin{equation}\label{varcomp}
\begin{split}
\Var(\mathcal Z_n^a[2])&=2\frac{n}{\mathcal N_n^2}\sum_{\lambda\in \Lambda_n^+}
\left ( 2 \int_0^L
\left \langle \frac{\lambda}{|\lambda|}, \dot \gamma(t)\right \rangle^2\,dt  -L \right)^2\cr
&= \frac{n}{\mathcal N_n^2}\sum_{\lambda\in \Lambda_n}
\left ( 2 \int_0^L
\left \langle \frac{\lambda}{|\lambda|}, \dot \gamma(t)\right \rangle^2\,dt  -L \right)^2\cr
&= \frac{n}{\mathcal N_n}\bigg(4\int_0^L \int_0^L \frac{1}{\mathcal N_n}\sum_{\lambda\in \Lambda_n}
\left \langle \frac{\lambda}{|\lambda|}, \dot \gamma(t_1)\right \rangle^2\left \langle \frac{\lambda}{|\lambda|}, \dot \gamma(t_2)\right \rangle^2\,dt_1 dt_2 \\&\new{- 4L\int_0^L \frac{1}{\mathcal N_n}\sum_{\lambda\in \Lambda_n}\left\langle \frac{\lambda}{|\lambda|}, \dot \gamma(t)\right \rangle^2dt}
+L^2   \bigg)\cr
&\new{=\frac{n}{\mathcal N_n} (4B_{\mathcal C}(\mu_n) - 2L^{2}+L^2)}=\frac{n}{\mathcal N_n} (4B_{\mathcal C}(\mu_n) - L^2),
\end{split}
\end{equation}
\new{where we used \eqref{eq:ident sum squares prod} with $z=\dot \gamma(t)$ to get from the $4$th to the $5$th line of \eqref{varcomp}}.
Next we evaluate the variance of $\mathcal Z_n^b[2]$: by \paref{eq3} it is given by
\begin{equation}\label{var2b}
\begin{split}
&\Var(\mathcal Z_n^b[2])
=\frac{n}{2{\mathcal N}_n^2}  \sum_{\lambda\ne \lambda',\lambda''\ne \lambda'''}
\E[a_\lambda \overline{a_{\lambda'}}a_{\lambda''} \overline{a_{\lambda'''}}] \times \\&\times\int_0^L \left(2\left \langle \frac{\lambda}{|\lambda|},
\dot \gamma(t)\right \rangle \left \langle \frac{\lambda'}{|\lambda|}, \dot \gamma(t)\right \rangle
 -1 \right)\e^{i2\pi\langle \lambda -\lambda', \gamma(t)\rangle}\,dt\times\cr
 &\times\int_0^L \left(2\left \langle \frac{\lambda''}{|\lambda''|},
\dot \gamma(t)\right \rangle \left \langle \frac{\lambda'''}{|\lambda'''|}, \dot \gamma(t)\right \rangle
 -1 \right)\e^{i2\pi\langle \lambda'' -\lambda''', \gamma(t)\rangle}\,dt\cr
 &\le \new{C\frac{n}{2{\mathcal N}_n^2}\sum_{\lambda\ne \lambda'}\left | \int_0^L\left(2\left \langle \frac{\lambda}{|\lambda|},
\dot \gamma(t)\right \rangle \left \langle \frac{\lambda'}{|\lambda|}, \dot \gamma(t)\right \rangle
 -1 \right)\e^{i2\pi\langle \lambda -\lambda', \gamma(t)\rangle}\,dt\right |^2},
\end{split}
\end{equation}
for some $C>0$.

Using the non-vanishing curvature assumption on $\Cc$, Van der Corput's Lemma (see e.g. \cite[Lemma 5.2]{R-W}) implies that for $\lambda\ne \lambda'$
the inner oscillatory integral on the r.h.s. of \eqref{var2b} may be bounded as
$$
\new{\left | \int_0^L\left(2\left \langle \frac{\lambda}{|\lambda|},
\dot \gamma(t)\right \rangle \left \langle \frac{\lambda'}{|\lambda|}, \dot \gamma(t)\right \rangle
 -1 \right)\e^{i2\pi\langle \lambda -\lambda', \gamma(t)\rangle}\,dt\right |} \ll \frac{1}{|\lambda-\lambda'|^{1/2}};
$$
this together with \paref{var2b} yield
\begin{equation}\label{good_estimate2}
\Var(\mathcal Z_n^b[2])\ll \frac{n}{\mathcal N_n^2} \sum_{\lambda\ne \lambda'} \frac{1}{|\lambda - \lambda'|}.
\end{equation}
Using the bound
\begin{equation*}
\sum_{\lambda\ne \lambda'} \frac{1}{|\lambda - \lambda'|} \ll_{\epsilon} \Nc_{n}^{\epsilon}
\end{equation*}
for every $\epsilon>0$ from \cite[Proposition 5.3]{R-W} to bound the r.h.s. of \paref{good_estimate2}, and comparing the result with \paref{varcomp} we have
\begin{equation}\label{opiccolo}
\Var(\mathcal Z_n^b[2])=o(\Var(\mathcal Z_n^a[2])).
\end{equation}
Now substituting \eqref{varcomp}, \eqref{opiccolo} into
\eqref{vv}, and using Cauchy-\new{Schwarz} to bound the covariance term in \paref{vv}, finally yield
\begin{equation*}
\Var(\mathcal Z_n[2]) \sim \Var(\mathcal Z_n[2]^a) = \frac{n}{\mathcal N_n} (4B_{\mathcal C}(\mu_n) - L^2).
\end{equation*}
This taking into account \eqref{igor_var}, concludes the proof of Proposition \ref{secondvar}.

\end{proof}

\begin{proof}[Proof of Proposition \ref{prop1} assuming Lemma \ref{lemproj2}]
Thanks to \paref{opiccolo},
it suffices to investigate the asymptotic distribution of $\mathcal Z_n^a[2]$ as in \paref{eq2}. We have
\begin{equation}\label{asymp-proj2}
\begin{split}
\frac{\mathcal Z_n^a[2]}{\sqrt{\Var(\mathcal Z_n^a[2])}}
= \frac{1}{\sqrt{\mathcal N_n/2}} \sum_{\lambda\in \Lambda_n^+}
(|a_\lambda|^2 -1) \frac{2 \int_0^L
\left \langle \frac{\lambda}{|\lambda|}, \dot \gamma(t)\right \rangle^2\,dt  -L }
{\sqrt{ 4B_{\mathcal C}(\mu_n) - L^2}}.
\end{split}
\end{equation}
Now we can apply Lindeberg's criterion (see e.g. \cite[Remark 11.1.2]{N-P}):
since, as $\mathcal N_n\to +\infty$, we have
$$
\max_{\lambda\in \Lambda_n^+} \frac{1}{\sqrt{\mathcal N_n/2}}\left | \frac{2 \int_0^L
\left \langle \frac{\lambda}{|\lambda|}, \dot \gamma(t)\right \rangle^2\,dt  -L }{\sqrt{ 4B_{\mathcal C}(\mu_n) - L^2}}     \right | \to 0,
$$
then
$$
\frac{\mathcal Z^a_n[2]}{\sqrt{\Var(\mathcal Z^a_n[2])}}\mathop{\to}^{d} Z,
$$
where $Z\sim \mathcal N(0,1)$.

\end{proof}

\section{Proof of Proposition \ref{kac-rice}}\label{SproofKac}

The proof of Proposition \ref{kac-rice} is inspired by the proof of Theorem 1.2 in \cite{R-W}, we refer the reader to \cite[\S 1.3]{R-W} and \cite[\S 1.5]{R-W-Y} for a complete discussion.
We will need to inspect the proof of the \emph{approximate} Kac-Rice formula \cite[Proposition 1.3]{R-W},
which gives the asymptotic variance of the nodal intersections number in terms of an explicit integral that involves the covariance function $r=r_n$ in \paref{covf}  and a couple of its derivatives
$$
r_1 := \frac{\partial}{\partial t_1} r,\quad  r_2 := \frac{\partial }{\partial t_2} r,\quad r_{12} := \frac{\partial^2}{\partial t_1 \partial t_2} r,
$$
and then study higher order terms.

\subsection{Auxiliary lemmas}

\new{The following lemma, proved in \S\ref{sec:kac-rice approx pr} below, referred to as ``Approximate Kac-Rice formula"
(due to the existence of the error term on the r.h.s. of \ref{kr-approx}),
and valid for static curves only, is a refined version of
\cite[Proposition 1.3]{R-W}, with more terms in the main expansion, and a smaller error term. Recall that
$\alpha$ is given by \eqref{alpha_var}.}

\begin{lemma}[``Approximate Kac-Rice"]\label{kac-rice approx}
Let $\mathcal C\subset \mathbb T$ be a static curve of total length $L$, and $\lbrace n\rbrace\subset S$ a $\delta$-separated sequence such that $\mathcal N_n\to +\infty$. Then the intersection number variance is asymptotic to
\begin{equation}\label{kr-approx}
\begin{split}
\Var(\mathcal Z_n) &= n\,\int_0^L \int_0^L \Big( \frac34 r^4+\frac{1}{12}(r_{12}/\alpha)^4
- \frac{(r_2/\sqrt{\alpha})^4}{4}  - \frac{(r_1/\sqrt{\alpha})^4}{4} \\&+2(r_{12}/\alpha)r (r_1/\sqrt{\alpha})( r_2/\sqrt{\alpha})+\frac{(r_1/\sqrt{\alpha})^2(r_2/\sqrt{\alpha})^2}{2} - \frac{3}{2}r^2(r_2/\sqrt{\alpha})^2 \\& - \frac{3}{2}r^2(r_1/\sqrt{\alpha})^2+\frac{1}{2} (r_{12}/\alpha)^2 r^2+\frac12 (r_2/\sqrt \alpha)^2 (r_{12}/\alpha)^2 \\&+\frac12 (r_1/\sqrt \alpha)^2 (r_{12}/\alpha)^2\Big)\,dt_1 dt_2
+o\left(\frac{n}{\mathcal N_n^2}\right).
\end{split}
\end{equation}
\end{lemma}
To find the asymptotics of the moments of $r$ and its derivatives appearing in the r.h.s. of \paref{kr-approx}, in particular to bound the contribution of ``off-diagonal" terms, we will need the following \new{lemma} whose proof is given in Appendix \ref{Soff-diagonal}.
\begin{lemma}\label{remainder4}
Assume that $\{ n\}\subset S$ is a $\delta$-separated sequence, then
\begin{equation*}
\min_{\substack{\lambda_1, \lambda_2,\lambda_3,\lambda_4\in \Lambda_n \\ \lambda_{1}+\lambda_{2}+\lambda_{3}+\lambda_{4}\ne 0}} \|\lambda_1 + \lambda_2 + \lambda_3 + \lambda_4\| \gg n^{2\delta},
\end{equation*}
where the constant involved in the `` $\gg$" notation is absolute. In particular,
\begin{equation}\label{quantity4_estim}
\frac{1}{\mathcal N_n^2} \sum_{{\substack{\lambda_1, \lambda_2,\lambda_3,\lambda_4\in \Lambda_n \\ \lambda_{1}+\lambda_{2}+\lambda_{3}+\lambda_{4}\ne 0}}} \frac{1}{\|\lambda_1 + \lambda_2 + \lambda_3 + \lambda_4\|} = o(1).
\end{equation}
\end{lemma}
Let us also introduce some more notation:
\begin{equation}\label{Fmu}
\begin{split}
F_{\mathcal C}(\mu_n):= \int_{0}^L \int_{0}^L \frac{1}{\mathcal N_n^2}\sum_{\lambda,\lambda'} \left \langle \frac{\lambda}{|\lambda|},\dot \gamma(t_1)\right \rangle \left \langle \frac{\lambda}{|\lambda|},\dot \gamma(t_2)\right \rangle
\left \langle \frac{\lambda'}{|\lambda'|},\dot \gamma(t_1)\right \rangle \left \langle \frac{\lambda'}{|\lambda'|},\dot \gamma(t_2)\right \rangle \,dt_1 dt_2.
\end{split}
\end{equation}

We can now state the following lemma\new{, whose proof will be given in Appendix \ref{Smoments} below}.

\begin{lemma}\label{lemma4}
If $\mathcal C\subset \T$ is a smooth curve with nowhere vanishing curvature, then for $\delta$-separated sequences $\lbrace n\rbrace$ such that $\mathcal N_n\to +\infty$, we have
\begin{equation}\label{moments4}
\begin{split}
&1) \int_{\mathcal C} \int_{\mathcal C} r(t_1,t_2)^4\,dt_1 dt_2
=3L^2 \frac{1}{\mathcal N_n^2}+o\left( \mathcal N_n^{-2}\right),\cr
&2) \int_{\mathcal C} \int_{\mathcal C} \left(\frac{1}{\sqrt{\alpha}}r_1(t_1,t_2)\right)^4\,dt_1 dt_2= 3 L^2\frac{1}{\mathcal N_n^2}  + o(\mathcal N_n^{-2}),\cr
&3) \int_{\mathcal C} \int_{\mathcal C} \left(\frac{1}{\alpha}r_{12}(t_1,t_2)\right)^4\,dt_1 dt_2
=2^4\cdot 3A_{\mathcal C}(\mu_n) \frac{1}{\mathcal N_n^2}+ o(\mathcal N_n^{-2}),\cr
&4)  \int_{\mathcal C} \int_{\mathcal C}(r_{12}/\alpha)r (r_1/\sqrt{\alpha})( r_2/\sqrt{\alpha})\,dt_1 dt_2=
-4F_{\mathcal C}(\mu_n) \frac{1}{\mathcal N_n^2} + o(\mathcal N_n^{-2}),\cr
&5)  \int_{\mathcal C} \int_{\mathcal C}(r_1/\sqrt{\alpha})^2(r_2/\sqrt{\alpha})^2\,dt_1 dt_2=
L^2\frac{1}{\mathcal N_n^2}+4\cdot 2\frac{1}{\mathcal N_n^2}F_{\mathcal C}(\mu_n)+ o(\mathcal N_n^{-2}),\cr
&6) \int_{\mathcal C} \int_{\mathcal C}r^2(r_2/\sqrt{\alpha})^2\,dt_1 dt_2
=\frac{1}{\mathcal N_n^2}L^2+ o(\mathcal N_n^{-2}),\cr
&7)\int_{\mathcal C} \int_{\mathcal C}r^2(r_{12}/\alpha)^2\,dt_1 dt_2
=4\frac{1}{\mathcal N_n^2} B_{\mathcal C}(\mu_n)
+8\frac{1}{\mathcal N_n^2} F_{\mathcal C}(\mu_n)+ o(\mathcal N_n^{-2}),\cr
&8)\int_{\mathcal C} \int_{\mathcal C}(r_1/\sqrt \alpha)^2(r_{12}/\alpha)^2\,dt_1 dt_2
=4\frac{1}{\mathcal N_n^2} B_{\mathcal C}(\mu_n) + o(\mathcal N_n^{-2}),
\end{split}
\end{equation}
where $B_{\mathcal C}(\mu_n)$, $A_{\mathcal C}(\mu_n)$ are given in \paref{B} and \paref{Amu} with $\mu=\mu_n$, respectively, and $F_{\mathcal C}(\mu_n)$ in \paref{Fmu}.
\end{lemma}

\subsection{Proof of Proposition \ref{kac-rice}}

\begin{proof}[Proof]
Upon substituting \paref{moments4} into \paref{kr-approx}, we obtain
\begin{equation}\label{daje}
\Var(\mathcal Z_n)= \frac{n}{4\mathcal N^2_n} \left ( 16 A_{\mathcal C}(\mu_n) + 24 B_{\mathcal C}(\mu_n) - 7L^2\right)  +o\left (\frac{n}{\mathcal N_n^2}\right )
\end{equation}
after some straightforward manipulations.
Since $\mathcal C$ is static, $4B_{\mathcal C}(\mu_n) = L^2$ so that \paref{daje} is
 \begin{equation*}
\Var(\mathcal Z_n)= \frac{n}{4\mathcal N^2_n} \left ( 16 A_{\mathcal C}(\mu_n) -L^2\right)  +o\left (\frac{n}{\mathcal N_n^2}\right ),
\end{equation*}
which is \paref{eqvar42}.

\vspace{2mm}

\new{

Now we prove that the leading term $16 A_{\mathcal C}(\mu_n) -L^2$ is bounded away from zero, assuming that $\Cc$ is static.
In what follows, we show that if $\Cc$ is static, then $16 A_{\mathcal C}(\mu) -L^2>0$ is strictly positive for
every probability measure on $\Sc^{1}$ invariant w.r.t. rotation by $\frac{\pi}{2}$; by the compactness of
$\frac{\pi}{2}$-invariant probability measures this is sufficient for our purposes.

By denoting the inner integral in \paref{Amu}
$$
B(\theta_1, \theta_2):=\int_0^L \langle \theta_1, \dot \gamma(t)  \rangle^2
\langle \theta_2, \dot \gamma(t)  \rangle^2\,dt,
$$
we may rewrite the definition \paref{Amu} of $A_{\mathcal C}(\mu)$ as
\begin{equation*}\label{AwithB}
A_{\mathcal C}(\mu) = \int_{\mathcal S^1} \int_{\mathcal S^1} B(\theta_1, \theta_2)^2  d\mu(\theta_1) d\mu(\theta_2),
\end{equation*}
For $\theta\in\Sc^{1}$ we denote $\theta^{\perp}$ to be its rotation by $\frac{\pi}{2}$.
Then, for every $\theta_1,\theta_2 \in \mathcal S^1$, we have
\begin{equation}\label{sumL}
B(\theta_1, \theta_2) + B(\theta_1^\perp, \theta_2) + B(\theta_1, \theta_2^\perp) + B(\theta_1^\perp, \theta_2^\perp)= L.
\end{equation}
Upon maximizing and minimizing the function $B(\theta_1, \theta_2)^2 + B(\theta_1^\perp, \theta_2)^2 + B(\theta_1, \theta_2^\perp)^2 + B(\theta_1^\perp, \theta_2^\perp)^2$ under the constraint \paref{sumL}, we get that
\begin{equation}\label{MaxMin}
\frac{L^2}{4}\le B(\theta_1, \theta_2)^2 + B(\theta_1^\perp, \theta_2)^2 + B(\theta_1, \theta_2^\perp)^2 + B(\theta_1^\perp, \theta_2^\perp)^2 \le L^2,
\quad \forall \theta_1, \theta_2\in \mathcal S^1.
\end{equation}

We can exploit the invariance property of the probability measure $\mu$ and of $B(\cdot,\,\cdot)$ w.r.t. sign changes, to rewrite \eqref{AwithB} as
$$
A_{\mathcal C}(\mu) = \int_{\mathcal S^1/i} \int_{\mathcal S^1/i}4 \left( B(\theta_1, \theta_2)^2 + B(\theta_1^\perp, \theta_2)^2 + B(\theta_1, \theta_2^\perp)^2 + B(\theta_1^\perp, \theta_2^\perp)^2\right)  \,d\mu(\theta_1)d\mu(\theta_2),
$$
where $\mathcal S^1/i$ is a quarter of the circle of measure $\mu(\mathcal S^1/i)=1/4$.
By \paref{MaxMin} we have
$$
\frac{L^2}{16} \le A_{\mathcal C}(\mu) \le \frac{L^2}{4},
$$
which implies
$$
0\le 16A_{\mathcal C}(\mu) - L^2\le 3L^2,
$$
and moreover, that
\begin{equation*}
A_{\mathcal C}(\mu)= \frac{L^{2}}{16},
\end{equation*}
if and only if for every $\theta_{1},\theta_{2}\in\supp \mu$ in the support of $\mu$, we have
$$
B(\theta_1, \theta_2)^2 + B(\theta_1^\perp, \theta_2)^2 + B(\theta_1, \theta_2^\perp)^2 + B(\theta_1^\perp, \theta_2^\perp)^2 = \frac{L^2}{4}.
$$
Further, given $\theta_{1},\theta_{2}\in \Sc^{1}$, the minimum in \eqref{MaxMin} is attained, i.e.
$$
B(\theta_1, \theta_2)^2 + B(\theta_1^\perp, \theta_2)^2 + B(\theta_1, \theta_2^\perp)^2 + B(\theta_1^\perp, \theta_2^\perp)^2 = \frac{L^2}{4},
$$
if and only if
$$
B(\theta_1, \theta_2) = B(\theta_1^\perp, \theta_2) = B(\theta_1, \theta_2^\perp) = B(\theta_1^\perp, \theta_2^\perp) = \frac{L}{4}.
$$

\vspace{3mm}

Now assume that for some $\mu$ invariant w.r.t. $\frac{\pi}{2}$-rotation,
we have $$A_{\mathcal C}(\mu)= \frac{L^{2}}{16},$$ so that, by the above, we have
$B(\theta_{0},\theta_{0})=\frac{L}{4}$ for some $\theta_{0}\in\Sc^{1}$ (any $\theta_{0}\in \supp\mu$ would do).
For a unit speed parametrization $\gamma:[0,L]\to \mathcal C$ we have $\dot \gamma(t) = \e^{i\phi(t)}$, where $\phi(t)$ is the angle of the derivative $\dot \gamma(t)$ w.r.t. the coordinate axis. Recalling standard trigonometric identities, we may write
\begin{equation}\label{develop}
\begin{split}
\frac{L}{4}=B(\theta_{0}, \theta_{0}) &= \int_0^L \cos^4(\phi(t)-\theta_{0})\,dt\cr
&=\frac38 L + \frac12 \int_0^L\cos(2(\phi(t)-\theta_{0}))\,dt +\frac18 \int_0^L\cos(4(\phi(t)-\theta_{0}))\,dt.
\end{split}
\end{equation}
From \cite[Corollary 7.2]{R-W} for static curves $\Cc$ we have
\begin{equation}\label{zero}
\int_0^L\cos(2(\phi(t)-\theta_{0}))\,dt=0,
\end{equation}
and substituting \paref{zero} into \paref{develop}, we obtain
\begin{equation}\label{app0}
\int_0^L \cos(4(\phi(t)-\theta_{0}))\,dt = -L.
\end{equation}
The equality \paref{app0} holds true if and only if for every $t\in [0,L]$
$$
\cos(4(\phi(t)-\theta_{0})) = -1,
$$
which certainly implies that $\mathcal C$ is a straight line segment,
which, by our conventions (Definition \ref{defstatic}), is a contradiction to our assumption that $\Cc$ is static.

}

\end{proof}

The rest of this section is dedicated to proving Lemma \ref{kac-rice approx}.

\subsection{Proof of Lemma \ref{kac-rice approx}}

\label{sec:kac-rice approx pr}

Our proof follows along the same path of thought as \cite[Proposition 1.3]{R-W}, except that we add one more term
in the expansion of the $2$-point correlation function \eqref{eqTayK2}, and need to control the error terms using the more difficult
to evaluate higher moments. Thereupon we are going to bring the essence of the proof, highlighting the differences, sometimes
omitting some details identical to both cases.

\subsubsection{Preliminaries}

The main idea is to divide the square $[0,L]^2$ into small sub-squares and apply the usual Kac-Rice formula to ``most" of them, bounding the contribution of the remaining terms (for an extensive discussion see \cite[\S 1.3]{R-W} and \cite[\S 1.5]{R-W-Y}).

Let us divide the interval $[0,L]$ into small sub-intervals of length roughly $1/\sqrt{E_n}$. To be more precise, let $c_0>0$ (chosen as in Lemma \ref{lemmasmall}) and  set $k:=\lfloor L \sqrt{E_n}/c_0 \rfloor+ 1$ and $\delta_0 := L/k$;
consider the sub-intervals, for $i=1,\dots , k$, defined as
$$
I_i := [(i-1)\delta_0, i\delta_0].
$$
Let us now denote by $\mathcal Z_i$ the number of zeros of $f_n$ in $I_i$, so that
$$
\mathcal Z_n = \sum_i \mathcal Z_i,
$$
and therefore
\begin{equation}\label{varsum}
\Var(\mathcal Z_n) = \sum_{i,j} \Cov(\mathcal Z_i, \mathcal Z_j).
\end{equation}
Let us define, for $i,j=1,\dots,k$ the squares
\begin{equation}\label{little_square}
S_{i,j} := I_i\times I_j = [(i-1)\delta_0, i\delta_0]\times [(j-1)\delta_0, j\delta_0],
\end{equation}
so that
$$
[0,L]^2 = \bigcup_{i,j} S_{i,j}.
$$
\begin{definition}\label{def_sing_set}{\rm (cf. \cite[Definition 4.5]{R-W} and \cite[Definition 2.7]{R-W-Y})}
We say that $S_{i,j} \subset [0,L]^2$ in \paref{little_square} is \emph{singular} if there exists $(t_1,t_2)\in S_{i,j}$ such that
$$
\new{|r(t_1,t_2)|} > 1/2.
$$
\end{definition}

\new{For example, all {\em diagonal} squares $S_{i,i}$ are singular, because $r(t,t)\equiv 1$.}
Since $r/\sqrt{E_n}$ is a Lipschitz function with absolute constant, if $S_{i,j}$ is singular, then
$$
\new{|r(t_1,t_2)|} > 1/4
$$
for every $(t_1,t_2)\in S_{i,j}$\new{, provided that $c_{0}$ is chosen sufficiently small.}
Upon invoking \paref{varsum} we may write
\begin{equation}\label{varsum2}
\Var(\mathcal Z_n) = \sum_{i,j : S_{i,j}\text{ non-sing.}} \Cov(\mathcal Z_i, \mathcal Z_j) + \sum_{i,j : S_{i,j}\text{ sing.}} \Cov(\mathcal Z_i, \mathcal Z_j).
\end{equation}

We apply the Kac-Rice formula on the non-singular squares and bound the contribution corresponding to the singular part as follows.

\begin{lemma}\label{lemmasing}
For a $\delta$-separated sequence $\lbrace n\rbrace \subset S$ such that $\mathcal N_n\to +\infty$, we have
\begin{equation}\label{sing2}
\left |\sum_{i,j : S_{i,j}\text{ sing.}} \Cov(\mathcal Z_i, \mathcal Z_j)\right | = o\left ( \frac{n}{\mathcal N_n^{2}}\right ).
\end{equation}
\end{lemma}

Lemma \ref{lemmasing} will be proved in \S\ref{Ssing}.
Let us now deal with the non-singular part.
If $S_{i,j}$ is non-singular (Definition \ref{def_sing_set}), then $r(t_1, t_2) \ne \pm 1$ for every $(t_1,t_2)\in S_{i,j}$ (note that necessarily $i\ne j$). We can apply Kac-Rice formula (\cite{A-W}, \cite[\new{Equality (3.2)}]{R-W}) to write
\begin{equation}\label{covkr}
\Cov(\mathcal Z_i, \mathcal Z_j) = \int_{I_i} \int_{I_j} \left (K_2(t_1, t_2) - K_{1}(t_1)K_{1}(t_2)\right )\,dt_1 dt_2,
\end{equation}
where (\cite[Lemma 2.1]{R-W})
$$
K_1(t) := \phi_{t}(0)\cdot \E[f'_n(t)|f_n(t)=0]=\sqrt{2} \sqrt{n},
$$
$\phi_t$ being the density of the Gaussian random variable $f_n(t)$.

The function
$K_2(t_1, t_2)$ is the $2$-point correlation function of zeros of the process $f_n$ (see \cite[\S 3.2]{R-W}), defined as follows: for $t_1\ne t_2$
$$
K_2(t_1, t_2) = \phi_{t_1,t_2}(0,0)\cdot   \E\left[ |f_n'(t_1)|\cdot |f_n'(t_2)| | f_n(t_1)=f_n(t_2)=0\right],
$$
$\phi_{t_1,t_2}$ being the probability density of the Gaussian vector $(f_n(t_1), f_n(t_2))$.
The function $K_2$ admits a continuation to a smooth function on the whole of \new{$[0,L]^{2}$} (see \cite{R-W}), though its values at the diagonal are of no significance for our purposes.

\subsubsection{Proof of Lemma \ref{kac-rice approx}}

First we need to Taylor expand the $2$-point correlation function, which will be proven in \S\ref{sectiontaylork2}.
\new{The following is a refinement of \cite[Proposition 3.2]{R-W}, valid for any curve.}

\begin{lemma}\label{taylorK2} For every $\varepsilon >0$, the two-point correlation function $K_2$
satisfies, uniformly for $|r|<1-\varepsilon$,
\begin{equation}\label{eqTayK2}
\begin{split}
K_2(t_1,t_2) = \frac{\alpha}{\pi^2} \Big( &1 +\frac{1}{2} (r_{12}/\alpha)^2 +\frac12 r^2   - \frac{(r_2/\sqrt{\alpha})^2}{2}   - \frac{(r_1/\sqrt{\alpha})^2}{2}
+\frac38 r^4+\frac{1}{24}(r_{12}/\alpha)^4\cr
&- \frac{(r_2/\sqrt{\alpha})^4}{8}  - \frac{(r_1/\sqrt{\alpha})^4}{8} +(r_{12}/\alpha)r (r_1/\sqrt{\alpha})( r_2/\sqrt{\alpha})\cr&+\frac{(r_1/\sqrt{\alpha})^2(r_2/\sqrt{\alpha})^2}{4}
 - \frac{3}{4}r^2(r_2/\sqrt{\alpha})^2  - \frac{3}{4}r^2(r_1/\sqrt{\alpha})^2+\frac{1}{4} (r_{12}/\alpha)^2 r^2\cr
&+\frac14 (r_2/\sqrt \alpha)^2 (r_{12}/\alpha)^2 +\frac14 (r_1/\sqrt \alpha)^2 (r_{12}/\alpha)^2\Big)\cr
&+\alpha O\left(r^6 + (r_1/\sqrt \alpha)^6 +  (r_2/\sqrt \alpha)^6  +  (r_{12}/ \alpha)^6  \right),
\end{split}
\end{equation}
where the constants involved in the $``O"$-notation depend only on $\varepsilon$.
\end{lemma}
We can now prove Lemma \ref{kac-rice approx}.

\begin{proof}[Proof of Lemma \ref{kac-rice approx}]
Substituting \paref{covkr} and \paref{sing2} into \paref{varsum2}, valid for non-singular sets, we may write
$$
\Var(\mathcal Z_n) = \int_{[0,L]^2\setminus B} \left (K_2(t_1, t_2) - K_{1}(t_1)K_{1}(t_2)\right )\,dt_1 dt_2 +o\left(n \mathcal N_n^{-2}\right),
$$
where $B$ denotes the union of all singular sets $S_{i,j}$ (see \paref{defBsing}).
Outside of $B$, we can use Lemma \ref{taylorK2} together with Lemma \ref{remainder6} to bound the error term. Lemma \ref{lemma2} under the assumption $4B_{\mathcal C}(\mu)=L^2$ \new{(and $\alpha$ given by \eqref{alpha_var})} gives
\begin{equation*}
\begin{split}
& \int_0^L \int_0^L \Big((r_{12}/\alpha)^2 + r^2   - (r_2/\sqrt{\alpha})^2  - (r_1/\sqrt{\alpha})^2 \Big)\,dt_1 dt_2 = o(\mathcal N_n^{-2}),
\end{split}
\end{equation*}
and the uniform boundedness of the integrand on the r.h.s. of \paref{kr-approx} allows to conclude the proof.

\end{proof}

\section{Proof of propositions \ref{prop-var4} and \ref{4chaos}}\label{Schaos4}

The following lemma is \eqref{var_chaos2} of Proposition \ref{prop-var4}.

\begin{lemma}\label{var2static}
Let $\mathcal C\subset \mathbb T$ be a static curve on the torus, of total length $L$. Let $\lbrace n\rbrace\subset S$ be a $\delta$-separated sequence such that $\mathcal N_n\to +\infty$, then
$$
\Var(\mathcal Z_n[2]) = o\left(\frac{n}{\mathcal N_n^{2}}   \right).
$$
\end{lemma}
\begin{proof}
First, since $\mathcal C$ is static, $\mathcal Z_n^a[2]\equiv 0$ by \paref{varcomp}.
Hence we have
$$
\mathcal Z_n[2] =  \mathcal Z_n^b [2],
$$
by Lemma \ref{lemproj2}.
Concerning $\mathcal Z_n^b$ we invoke \paref{good_estimate2} to bound
$$
\Var\left( \mathcal Z_n^b [2]  \right)\ll \frac{n}{\mathcal N_n^{2}} \sum_{\lambda\ne \lambda'} \frac{1}{|\lambda - \lambda'|}
\ll  \frac{n}{\mathcal N_n^{2}} \cdot \frac{\Nc_{n}^{2}}{\new{n^{1/4+\delta}}} = o\left( \frac{n}{\mathcal N_n^{2}}  \right),
$$
since our sequence of energy levels is assumed to be $\delta$-separated \paref{generic}.

\end{proof}

In what follows we study the limiting distribution of $\mathcal Z_n[4]$. We have the following explicit formula for the fourth chaotic component
$\mathcal Z_n[4]$ upon substituting $q=2$ in \eqref{proj}:
\begin{equation}\label{proj4}
\begin{split}
\mathcal Z_n[4] = \sqrt{2\pi^2n}\Big ( &b_4 a_0  \int_0^L H_{4}(f_n(t))\,dt \cr
&+b_2 a_2 \int_0^L H_{2}(f_n(t)) H_{2}(\widetilde f_n'(t))\,dt  +b_0a_4 \int_0^L  H_{4}(\widetilde f_n'(t))\,dt    \Big ).
\end{split}
\end{equation}

\subsection{Preliminaries}

Let us define the random variables ($n\in S$)
\begin{equation}\label{W1}
W_1(n) := \frac{1}{\sqrt{\mathcal N_n/2}}\sum_{\lambda\in \Lambda_n^+} (|a_\lambda|^2 -1),
\end{equation}
and the random processes
\begin{equation}\label{W}
W_2^t(n):=\frac{1}{\sqrt{\mathcal N_n/2}}\sum_{\lambda\in \Lambda_n^+} (|a_\lambda|^2 -1)
\, 2\left\langle \frac{\lambda}{|\lambda|}, \dot \gamma(t)\right\rangle^2.
\end{equation}
on $t\in [0,L]$.
In Lemma \ref{lemmachaos4} below we will find an explicit formula for $\Zc_{n}[4]$ in terms of the $W_{1}(n)$ and $W_{2}(n)$ and their by-products. To this end we first express each of the three terms in \paref{proj4} in terms of $W_{1}(n)$ and $W_{2}(n)$ (proven in Appendix \ref{appendix4}).

\begin{lemma}\label{H1}
For a $\delta$-separated sequence $\lbrace n \rbrace\subset S$ such that $\mathcal N_n\to +\infty$, we
have
\begin{equation*}
\begin{split}
&\int_0^L H_{4}(f_n(t))\,dt =\ X_n^a + X_n^b,\cr
&\int_0^L H_{4}(f'_n(t))\,dt
=Y_n^a + Y_n^b,\cr
&\int_0^L H_{2}(f_n(t))H_{2}(f'_n(t))\,dt
=  Z_n^a + Z_n^b ,
\end{split}
\end{equation*}
where for $n\in S$
\begin{equation}\label{xyz}
\begin{split}
&X_n^a := \frac{6L}{\mathcal N_n}(W_1(n)^2  - 1),\cr
&Y_n^a:=\frac{6L}{\mathcal N_n} \left (\int_0^L W_2^t(n)^2\,dt  - 4 \frac{1}{\mathcal N_n}\sum_{\lambda}
\int_0^L \left\langle \frac{\lambda}{|\lambda|}, \dot \gamma(t) \right\rangle^4\,dt \right),\cr
&Z_n^a:= \frac{2}{\mathcal N_n} \left (W_1(n)
\int_0^L  W_2^t(n)\,dt - L \right),
\end{split}
\end{equation}
and as $\mathcal N_n\to +\infty$
\begin{equation}\label{smallvariance}
\Var(X_n^b) = o\left (\frac{1}{\mathcal N_n^2} \right ), \Var(Y_n^b) = o\left (\frac{1}{\mathcal N_n^2} \right ), \Var(Z_n^b) = o\left (\frac{1}{\mathcal N_n^2} \right ).
\end{equation}
\end{lemma}
\mnew{(Note that the random variables $X_n^a, Y_n^a, Z_n^a$ defined in \paref{xyz} are centered.)}
We then have the following result.
\begin{lemma}\label{lemmachaos4}
For a $\delta$-separated sequence $\lbrace n \rbrace\subset S$ such that $\mathcal N_n\to +\infty$, we have
\begin{equation}\label{4ab}
\begin{split}
\mathcal Z_n[4] = \mathcal Z_n^a [4] + \mathcal Z_n^b [4],
\end{split}
\end{equation}
with
\begin{equation}\label{eq_proj4}
\begin{split}
\mathcal Z_n^a[4]
= \frac{\sqrt{2n}}{24}\Big(3 X^a_n  -Y^a_n  -6Z^a_n\Big ),
\end{split}
\end{equation}
where $X^a_n, Y^a_n$ and $Z^a_n$ are as in \paref{xyz},
\begin{equation}\label{goodvar}
\Var(\mathcal Z_n^a[4])\sim \frac{n}{4\mathcal N^2_n} \Big (16A_{\mathcal C}(\mu_n)+24B_{\mathcal C}(\mu_n) -7L^2 \Big ),
\end{equation}
and
\begin{equation}\label{varbsmallo}
\Var(\mathcal Z_n^b[4]) = o\left (\frac{n}{\mathcal N_n^2} \right ).
\end{equation}
\end{lemma}
\begin{proof}[Proof of Lemma \ref{lemmachaos4} assuming Lemma \ref{H1}]
We have the explicit values
\begin{equation}\label{absmall}
\begin{split}
&a_0=\sqrt{\frac{2}{\pi}},\quad a_2 = \sqrt{\frac{2}{\pi}} \frac{1}{2},\quad a_4 = \sqrt{\frac{2}{\pi}} \frac{-1}{2\cdot 2^2\cdot 3},\cr
&b_0= \frac{1}{\sqrt{2\pi}},\quad b_2 = \frac{-1}{2\sqrt{2\pi}},\quad b_4 = \frac{3}{4!\sqrt{2\pi}}
\end{split}
\end{equation}
by \paref{b} and \paref{a}.
Substituting \paref{absmall} into \paref{proj4}, we have
\begin{equation*}
\begin{split}
\mathcal Z_n[4]
&=\frac{\sqrt{2\pi^2n}}{\pi}\Big ( \frac{3}{4!}  \int_0^L H_{4}(f_n(t))\,dt -\frac14 \int_0^L H_{2}(f_n(t)) H_{2}(\widetilde f_n'(t))\,dt  -\frac{1}{4!}\int_0^L  H_{4}(\widetilde f_n'(t))\,dt    \Big )\cr
&=  \frac{\sqrt{2n}}{24}\Big(3 X^a_n  -Y^a_n  -6Z^a_n+  3X^b_n  -Y^b_n  -6Z^b_n \Big ),
\end{split}
\end{equation*}
by Lemma \ref{H1}; the latter is \paref{4ab} with $$\mathcal Z_n^b[4] := \frac{\sqrt{2n}}{24}\left( 3 X^b_n  -Y^b_n  -6Z^b_n \right ),$$
and \eqref{varbsmallo} follows from \eqref{smallvariance}.

Let us now prove \paref{goodvar}. First, observe that we can write
\begin{equation}\label{expr1}
\begin{split}
&3 X^a_n  -Y^a_n  -6Z^a_n= \frac{6}{\mathcal N_n} \Big[\frac{1}{\mathcal N_n/2}\sum_{\lambda,\lambda'\in \Lambda_n^+}(|a_\lambda|^2 -1)(|a_\lambda'|^2 -1)\times
\\&\times\int_0^L \left(3 -4 \left\langle \frac{\lambda}{|\lambda|}, \dot \gamma(t)\right\rangle^2 -4\left\langle \frac{\lambda}{|\lambda|}, \dot \gamma(t)  \right\rangle^2 \left\langle \frac{\lambda'}{|\lambda'|}, \dot \gamma(t)  \right\rangle^2   \right)\,dt\cr
&-L + 4 \frac{1}{\mathcal N_n}\sum_{\lambda\in \Lambda_n}
\int_0^L \left\langle \lambda, \dot \gamma(t)\right\rangle^4\,dt \Big].
\end{split}
\end{equation}
Equality \paref{expr1} and some straightforward computations yield
\begin{equation}\label{var1}
\Var(3 X^a_n  -Y^a_n  -6Z^a_n) = \frac{36}{\mathcal N_n^2} (A_n + B_n + 24 C_n),
\end{equation}
where
\begin{equation}\label{ABC}
\begin{split}
A_n &=  \frac{1}{\mathcal N_n^2/4}\sum_{\lambda,\lambda'\in \Lambda_n^+}
\left | \int_0^L\left(3 -4
\left\langle \frac{\lambda}{|\lambda|}, \dot \gamma(t)\right\rangle^2 -4\langle \frac{\lambda}{|\lambda|}, \dot \gamma(t)  \rangle^2
\left\langle \frac{\lambda'}{|\lambda'|}, \dot \gamma(t)  \right\rangle^2   \right)\,dt \right|^2,\cr
B_n &= \frac{1}{\mathcal N_n^2/4}\sum_{\lambda,\lambda'\in \Lambda_n^+}\int_0^L\left(3 -4 \left\langle \frac{\lambda}{|\lambda|}, \dot \gamma(t)\right\rangle^2 -4\left\langle \frac{\lambda}{|\lambda|}, \dot \gamma(t)  \right\rangle^2
\left\langle \frac{\lambda'}{|\lambda'|}, \dot \gamma(t)  \right\rangle^2   \right)\,dt\cr
&\times \int_0^L \left(3 -4 \left\langle \frac{\lambda'}{|\lambda'|}, \dot \gamma(s)\right\rangle^2 -
4\left\langle \frac{\lambda'}{|\lambda'|}, \dot \gamma(s)  \right\rangle^2
\left\langle \frac{\lambda}{|\lambda|}, \dot \gamma(s)  \right\rangle^2   \right)\,ds, \cr
C_n &= \frac{1}{\mathcal N_n^2/4}\sum_{\lambda\in \Lambda_n^+}
\left| \int_0^L \left(3 -4 \left\langle \frac{\lambda}{|\lambda|}, \dot \gamma(t)\right\rangle^2 -4\left\langle \frac{\lambda}{|\lambda|}, \dot \gamma(t)  \right\rangle^4  \right)\,dt\right|^2.
\end{split}
\end{equation}

Now, squaring out the respective terms on the \new{r.h.s.} of \paref{ABC}, with some straightforward computations we obtain
\begin{equation}\label{A2}
\begin{split}
A_n = -9L^2 +32 B_{\mathcal C}(\mu_n)   +16A_{\mathcal C}(\mu_n),
\end{split}
\end{equation}
\begin{equation}\label{B2}
\begin{split}
B_n = -5L^2 +16B_{\mathcal C}(\mu_n)+16A_{\mathcal C}(\mu_n),
\end{split}
\end{equation}
and that
\begin{equation}\label{C2}
\begin{split}
C_n = O\left( \frac{1}{\mathcal N_n}  \right).
\end{split}
\end{equation}
Substituting \paref{A2}, \paref{B2} and \paref{C2} into \paref{var1} we obtain
\begin{equation}\label{nice}
\begin{split}
\Var(3 X^a_n  -Y^a_n  -6Z^a_n) =\frac{36}{\mathcal N_n^2}(-14L^2 +48 B_{\mathcal C}(\mu_n) + 32 A_{\mathcal C}(\mu_n)  + o(1)),
\end{split}
\end{equation}
which, in turn, yields \eqref{goodvar}, bearing in mind \eqref{eq_proj4} with \eqref{nice}.

\end{proof}

\begin{proof}[Proof of Proposition \ref{prop-var4}]

First, \paref{var_chaos2} is the statement Lemma \ref{var2static}, and the leading term $16A_{\mathcal C}(\mu_n) - L^2$ was shown to be bounded away from zero as part of Proposition \ref{kac-rice}.

We now turn to proving \eqref{var_chaos4}.
Since the curve is assumed to be static, we have
$
4B_{\mathcal C}(\mu_n) = L^2
$. From \paref{goodvar} we have
\begin{equation*}
\begin{split}
\Var(\mathcal Z_n[4])
=\frac{n}{4\mathcal N^2_n} \left(
16A_{\mathcal C}(\mu_n)  -L^2\right ) + o\left( \frac{n}{\mathcal N_n^2}  \right),
\end{split}
\end{equation*}
which is \paref{var_chaos4}.

\end{proof}

\subsection{Proof of Proposition \ref{4chaos}}\label{subsec4}

\subsubsection{Auxiliary results}

Lemma \ref{lemmachaos4} implies that for $\delta$-separated sequences $\lbrace n\rbrace$ such that $\mathcal N_n\to +\infty$,
\begin{equation}
\frac{\mathcal Z_n[4]}{\sqrt{\Var(\mathcal Z_n[4])}} = \frac{\mathcal Z^a_n[4]}{\sqrt{\Var(\mathcal Z^a_n[4])}} + o_{\mathbb P}(1),
\end{equation}
where $\mathcal Z^a_n[4]$ is defined in \paref{eq_proj4} and $o_{\mathbb P}(1)$ denotes a sequence of random variables converging to $0$ in probability. We may then infer results on the limit distribution of $\Zc_{n}[4]$ from the corresponding results on $\Zc_n^a[4]$ for static curves. 

\begin{lemma}\label{lemmaLC}
If $4B_{\mathcal C}(\mu_n)-L^2 =0$, then for $W_1(n)$ and $W_2(n)$ defined in \paref{W1} and \paref{W} respectively we have
the identity
\begin{equation}\label{LC}
 \frac{1}{L} \int_0^L W_2^t(n)\,dt = W_1(n).
\end{equation}
\end{lemma}

\begin{proof}
The covariance matrix $\Sigma(n)$ of the Gaussian random vector $\left(W_1(n), \int_0^L W_2^t(n)\,dt \right)$
$$
\Sigma(n) = \begin{pmatrix}
&1 &L \\
&L &4B_{\mathcal C}(\mu_n)  \end{pmatrix}
$$
satisfies
$$\det \Sigma(n) = 4B_{\mathcal C}(\mu_n) -L^2=0.$$ Therefore $W_{1}(n)$ is a multiple of $\int_0^L W_2^t(n)\,dt ,$ which one may evaluate as given by \paref{LC}.

\end{proof}
\mnew{\begin{remark}\rm
 Lemma \ref{lemmaLC} can be alternatively proved as follows. First note that for $\mathcal C$ a static curve, for every $\theta\in [0,2\pi]$, from \paref{zero} it holds that
\begin{equation}\label{identity_static}
\frac{2}{L}\int_0^L \cos^2(\phi(t)-\theta)\,dt = 1,
\end{equation}
where $\gamma:[0,L]\to \mathbb T$ is a unit speed parameterization of $\mathcal C$ and $\dot \gamma(t) = \e^{i\phi(t)}$ as in \paref{ang_vel}.
Keeping in mind the definition \new{\paref{W}} of $W_2$ and the definition \new{\paref{W1}} of $W_1$ we write
\begin{equation*}\label{new_chain}
\begin{split}
\frac{1}{L} \int_0^L W_2^t(n)\,dt &= \frac{1}{\sqrt{\mathcal N_n/2}} \sum_{\lambda\in \Lambda_n^+} (|a_\lambda|^2 -1 ) \frac{2}{L}\int_0^L \cos^2(\phi(t) - \theta_{\lambda})\,dt\\
& = \frac{1}{\sqrt{\mathcal N_n/2}} \sum_{\lambda\in \Lambda_n^+} (|a_\lambda|^2 -1 )  = W_1(n),
\end{split}
\end{equation*}
with $\theta_\lambda$ denoting the argument of $\lambda/|\lambda|$ viewed as a complex number ($\lambda/|\lambda| =: \e^{i\theta_\lambda})$, and where
we used
\paref{identity_static}.
\end{remark}}

\begin{lemma}\label{intermediate4}
For a static curve $\mathcal C\subset \mathbb T$ we have
\begin{equation}\label{4n}
\begin{split}
\mathcal Z^a_n[4]
&= \frac{\sqrt{2n}}{4\mathcal N_n}\bigg(-\int_0^L \left(W_2^t(n) - \frac{1}{L}\int_0^L W_2^u(n)\,du \right)^2\,dt \\&+  4\frac{1}{\mathcal N_n}\sum_{\lambda_\in \Lambda_n}\int_{0}^L  \left \langle \frac{\lambda}{|\lambda|}, \dot \gamma(t)\right \rangle^4\,dt - L\bigg ).
\end{split}
\end{equation}
\end{lemma}

Lemma \ref{intermediate4} follows directly from \paref{eq_proj4} and Lemma \ref{lemmaLC} and \new{its proof} is omitted here. In order to study the asymptotic distribution of the r.h.s. of \paref{4n}, we need to study the asymptotic behavior of the sequence of stochastic processes $\lbrace W_2(n)\rbrace_n$, when $\mu_n\Rightarrow \mu$. The natural candidate to be the limiting process is the centred Gaussian process $W_2(\mu)=\lbrace W_2^t(\mu) \rbrace_t$ on $[0,L]$ uniquely defined by the covariance function
\mnew{\begin{equation}\label{defW2mau}
k_\mu(s,t) := \E[W_2^s(\mu)\cdot W_2^t(\mu)]=4\int_{\mathcal S^1} \langle \theta, \dot \gamma(t)\rangle^2\langle \theta, \dot \gamma(s)\rangle^2\,d\mu(\theta),\quad s,t\in [0,L].
\end{equation}}
The kernel $k_\mu$ above is positive-definite, hence the existence of such a $W_2(\mu)$ is guaranteed by the virtue of Kolmogorov's Theorem.

\begin{proposition}\label{mainth2}
Let $\mathcal C\subset \mathbb T$ be a static curve of length $L$, and $\lbrace n\rbrace\subset S$ a $\delta$-separated sequence
such that $\mathcal N_n\to +\infty$, and $\mu_n \Rightarrow \mu$. Then
\begin{equation}\label{limbello}
\frac{\mathcal Z^a_n[4]}{\sqrt{\Var(\mathcal Z^a_n[4])}}  \mathop{\to}^{d}  \mathcal I(\mu),
\end{equation}
where
\begin{equation}\label{I}
\mathcal I(\mu) := \frac{-\int_0^L \left(W_2^t(\mu) - \frac{1}{L}\int_0^L W_2^u(\mu)\,du \right)^2\,dt + \left( 4\int_{\mathcal S^1}\int_{0}^L  \langle \theta, \dot \gamma(t)\rangle^4\,dt d\mu(\theta)- L\right) }{\sqrt{16 A_{\mathcal C}(\mu) - L^2}},
\end{equation}
and $A_{\mathcal C}(\mu)$ is as in \paref{Amu}.
\end{proposition}
\begin{proof}
Thanks to Lemma \ref{intermediate4} and \paref{goodvar} for static curves, it suffices to prove that the stochastic processes
$ W_2(n)$ weakly converge to
$W_2(\mu)$. \mnew{Lindeberg's criterion (see \cite[Remark 11.1.2]{N-P}) \new{again}, and \new{\cite[Theorem 6.2.3]{N-P}}} ensure that the finite dimensional distributions of $W_2(n)$ converge to those of $W_2(\mu)$, so that a standard application of Prokhorov's Theorem (see e.g. \cite{Du}) allows to conclude the proof.

\end{proof}

\subsubsection{Proof of Proposition \ref{4chaos}}

Before giving a proof for Proposition \ref{4chaos}, we need to introduce some more notation.
Let us think of a probability measure $\mu$ on $\mathcal S^1$ as a probability measure on $[0,2\pi]$.  There exists a centered Gaussian process $\widetilde B=\widetilde B(\mu)$ indexed by $[0,2\pi]$ such that
\begin{equation}\label{kernelB}
\Cov\left ( \int_0^{2\pi} 1_A(a)\,d\widetilde B_a , \int_0^{2\pi}  1_B(a)\,d\widetilde B_a\right ) =  \mu(A\cap B),
\end{equation}
for any $A,B\in \mathcal B([0,2\pi])$ -- the Borel $\sigma$-field on the interval $[0,2\pi]$, $1_E$ denoting the indicator function of the set $E\in \mathcal B([0,2\pi])$.

Let us also introduce the following three centred, jointly Gaussian, random variables
\begin{equation}\label{defN}
\begin{split}
N_1 :=\int_0^{2\pi} (\cos a)^2\,d\widetilde B_a,\quad
N_2 := \int_0^{2\pi}  (\sin a)^2\,d\widetilde B_a,\quad
N_3 :=  \int_0^{2\pi} \cos a \cdot \sin a\,d\widetilde B_a,
\end{split}
\end{equation}
defined as stochastic integrals on $[0,2\pi]$ with respect to $\widetilde B$.

\begin{proof}[Proof of Proposition \ref{4chaos}]
First, let us denote by $G(d\mu)$ a Gaussian measure on $\mathcal S^1$ with control $\mu$ (see e.g. \cite{N-P}), i.e. a centered Gaussian family
$$
G= \lbrace G(A): A\in \mathcal B(\mathcal S^1)\rbrace
$$
such that
$$
\E[G(A)\cdot G(B)]=\mu(A\cap B).
$$
We have the following equality in law of stochastic processes
\begin{equation}\label{int}
W_2^{\new{t}} = 2\int_{\mathcal S^1} G(\mu(d\theta)) \langle \theta, \dot \gamma(\new{t} )\rangle^2,
\end{equation}
\mnew{where the l.h.s. of \paref{int} is the centered Gaussian process $W_2 = W_2(\mu)$ on $[0,L]$ defined by \paref{defW2mau}} and the r.h.s. of \paref{int} denotes the Wiener-It\^o integral on the unit circle with respect to the Gaussian measure $G(d\mu)$.

From \paref{int} we deduce that
\begin{equation}\label{int2}
W_2^t -\frac{1}{L}\int_0^L W_2^u\,du=2\int_{\mathcal S^1} G(\mu(d\theta)) \left(\langle \theta, \dot \gamma(t)\rangle^2 -\frac{1}{L}\int_0^L \langle \theta, \dot \gamma(u)\rangle^2\,du\right),
\end{equation}
again equality in law.
We parameterize the unit circle as $$[0,2\pi] \ni a\mapsto (\cos a, \sin a)\in \mathcal S^1.$$

Recalling \paref{kernelB} and \paref{defN}, we have by \paref{int2}
\begin{equation}\label{Nbello}
\begin{split}
&W_2^t -\frac{1}{L}\int_0^L W_2^u\,du\cr
&\mathop{=}^{d} 2\int_0^{2\pi} d\widetilde B_a \left((\cos a \cdot \dot \gamma_1(t) + \sin a\cdot \dot \gamma_2(t))^2 -\frac{1}{L}\int_0^L (\cos a \cdot \dot \gamma_1(u) + \sin a\cdot \dot \gamma_2(u))^2\,du\right)\cr
&=2\int_0^{2\pi} d\widetilde B_a (\cos a)^2 \left(\dot \gamma_1(t)^2  -\frac{1}{L}\int_0^L  \dot \gamma_1(u)^2\,du \right)\cr
& + 2\int_0^{2\pi} d\widetilde B_a (\sin a)^2 \left(\dot \gamma_2(t)^2  -\frac{1}{L}\int_0^L  \dot \gamma_2(u)^2\,du \right)\cr
&+4 \int_0^{2\pi} d\widetilde B_a \cos a \cdot \sin a \left(\dot \gamma_1(t)\dot \gamma_2(t)  -\frac{1}{L}\int_0^L  \dot \gamma_1(u)\dot \gamma_2(u)\,du \right)\cr
&=2N_1 f(t) - 2N_2 f(t) + 4N_3 g(t),
\end{split}
\end{equation}
where $f$ and $g$ are as in \paref{defF}.
The covariance matrix of $N:=(N_1,N_2,N_3)$ is
\begin{equation*}
\Sigma_N := \left ( \begin{matrix}
&\frac{3+\widehat{\mu}(4)}{8} & \frac{1-\widehat{\mu}(4)}{8} & 0\\
& \frac{1-\widehat{\mu}(4)}{8}&\frac{3+\widehat{\mu}(4)}{8} &0\\
&0 &0 &\frac{1-\widehat{\mu}(4)}{8}
\end{matrix}\right ).
\end{equation*}
 Note that $N_3$ is independent of $N_1$ and $N_2$.

The eigenvalues of $\Sigma_N$ are $\frac{1+\widehat{\mu}(4)}{4} , \frac12$ and $\frac{1-\widehat{\mu}(4)}{8} $, it is then immediate that
\begin{equation}\label{diagon}
\left ( \begin{matrix} &N_1\\ &N_2\\ &N_3
\end{matrix}\right ) \mathop{=}^{d} \left ( \begin{matrix} &\frac12 Z_1 + \frac{1}{\sqrt{2}}\sqrt{\frac{1+\widehat{\mu}(4)}{4}}Z_2\\ &\frac12 Z_1 - \frac{1}{\sqrt{2}}\sqrt{\frac{1+\widehat{\mu}(4)}{4}}Z_2\\ &\sqrt{\frac{1-\widehat{\mu}(4)}{8}}Z_3
\end{matrix}\right ),
\end{equation}
where $Z_1, Z_2, Z_3$ are i.i.d. standard Gaussian random variables. Substituting \paref{diagon} into \paref{Nbello}, thanks to Proposition \ref{mainth2}, we
can conclude the proof.

\end{proof}

\appendix

\mnew{\section{On Wiener-\^Ito chaotic decompositions}\label{appendix_chaos}

In this \new{section} we give some background on Wiener-\^Ito chaotic decompositions, for more details see e.g. \cite{N-P}.
The main idea of this decomposition relies on properties of Hermite polynomials $H_k, k=0, 1, 2, \dots$ \cite[\S 5.5]{Sz75}.
The latter are defined recursively as follows: $H_0 \equiv 1$, and, for $k\geq 1$,
 $$
H_{k}(t) = tH_{k-1}(t) - H'_{k-1}(t),\qquad t\in \R.
$$
For instance $H_1(t)=t, H_2(t) = t^2 - 1, H_3(t) = t^3 -3t, H_4(t) = t^4 - 6t^2 + 3.$
The normalized sequence of Hermite polynomials
\begin{equation}\label{H}
\{H_k/\sqrt{k!}, k\ge 0\} =: \mathbb{H}
\end{equation}
 is a complete orthonormal basis for the space $L^2(\mathbb{R}, \mathscr{B}(\mathbb{R}), \varphi(t)\,dt) =: L^2(\varphi)$
of square integrable functions on the real line w.r.t. the Gaussian standard density $\varphi(t) := e^{-t^2/2}/\sqrt{2\pi}$.

The arithmetic random waves $T_n$ defined in \S \ref{sec:arith rand wav}, and consequently their restriction to the curve $\mathcal C\subset \mathbb T$, i.e. the random process $f_n := T_n \circ \gamma$ given in \paref{process}, are a by-product of a family of complex Gaussian random variables $\{a_\lambda : \lambda\in \mathbb{Z}^2\}$ -- defined on some probability space $(\Omega, \mathscr{F}, \mathbb{P})$ -- such that
\begin{enumerate}
\item $a_\lambda = x_\lambda+iy_\lambda$, where $x_\lambda$ and $y_\lambda$ are two independent (real) Gaussian random variables with mean zero and variance $1/2$;
\item $a_\lambda$ and $a_{\lambda'}$ are  independent whenever $\lambda \notin\{ \lambda', -\lambda'\}$;
\item $a_\lambda = \overline{a_{-\lambda}}.$
\end{enumerate}
 Let us now define the space ${\bf A}$ to be the closure in $L^2(\mathbb{P})$ of all real finite linear combinations of random variables $\xi$ of the form $\xi = z \, a_\lambda + \overline{z} \, a_{-\lambda},$ where $\lambda\in \mathbb Z^2$ and $z\in \mathbb C$.
The space ${\bf A}$ is a real centered Gaussian Hilbert subspace
 $\subseteq L^2(\mathbb{P})$.

We are now in a position to recall the definition of Wiener chaoses. Let $q\ge 0$ be an integer, the $q$-th Wiener chaos $C_q$ associated with ${\bf A}$ is the closure (in $L^2(\mathbb{P})$) of all real finite linear combinations of random variables of the form
$$
H_{p_1}(\xi_1)\cdot H_{p_2}(\xi_2)\cdots H_{p_k}(\xi_k)
$$
for $k\ge 1$, where $p_1,...,p_k \geq 0$ are integers satisfying $$p_1+\cdots+p_k = q,$$
 and $(\xi_1,...,\xi_k)$ is a standard real Gaussian vector extracted
from ${\bf A}$ (in particular $C_0 = \mathbb{R}$ since $H_0\equiv 1$).

Taking advantage of the orthonormality and completeness properties of $\mathbb{H}$  \paref{H} in $L^2(\varphi)$, together with a standard monotone class argument (see e.g. \cite[Theorem 2.2.4]{N-P}), we can show that $$C_q \,\bot\, C_m, \qquad q\neq m,$$
where the orthogonality holds in the sense of $L^2(\mathbb{P})$, and
\begin{equation}\label{wiener}
L^2(\Omega, \sigma({\bf A}), \mathbb{P}) = \bigoplus_{q=0}^\infty C_q.
\end{equation}
The orthogonal decomposition of $L^2(\Omega, \sigma({\bf A}), \mathbb{P})$ in \paref{wiener} is equivalent to the following: every
\new{square-summable} real-valued functional $F$ \new{on} ${\bf A}$ can be (uniquely) represented in the form
\begin{equation}\label{e:chaos2}
F = \sum_{q=0}^\infty {\rm proj}(F \, | \, C_q)=\sum_{q=0}^\infty F[q],
\end{equation}
where $F[q]:={\rm proj}(F \, | \, C_q)$ is the orthogonal projection onto the $q$-th Wiener chaos $C_q$, and the series converges in $L^2(\mathbb{P})$. In particular, $F[0]={\rm proj}(F \, | \, C_0) = \E [F]$.

\smallskip

In this work we are interested in the decomposition \paref{e:chaos2} for $F=\mathcal Z_n$, the nodal intersection number given in \S \ref{sec:arith rand wav}, and formally represented as an explicit functional of the underlying Gaussian field $T_n$ in \paref{formal}. Keeping in mind the latter formula, it is worth noting that the random processes $f_n(t)$ and $f'_n(t) = \langle \nabla T_n(\gamma(t)), \dot \gamma(t)\rangle$ viewed as collections of random variables indexed by $t\in [0,L]$, are all lying in ${\bf A}$.
We decide to not give all the technical details to find the chaotic expansion of $\mathcal Z_n$.
\new{For a more complete discussion on the decomposition \paref{e:chaos2}, including all the technical details omitted in this manuscript, the interested reader might refer to \cite{K-L}; a brief account is given in \S\ref{Schaos} above.}

}

\section{Computations for the $2$nd chaotic component}\label{Schaos2}

\begin{proof}[Proof of Lemma \ref{lemproj2}]
From \paref{proj} with $q=1$ we have
\begin{equation}\label{proj2}
\mathcal Z_n[2]=\sqrt{2\pi^2n} \left(  b_2a_0 \int_0^L H_2(f_n(t))\,dt
 + b_0 a_2 \int_0^L H_2(\widetilde f_n'(t))\,dt\right).
\end{equation}
We evaluate the first summand in the r.h.s. of \paref{proj2} to be
\begin{equation}\label{1sum}
\begin{split}
\int_0^L H_2(f_n(t))\,dt &= \int_0^L (f_n(t)^2 -1)\,dt= \int_0^L (T_n(\gamma(t))^2 -1)\,dt\cr
&=\frac{1}{\mathcal N_n} \sum_{\lambda, \lambda'\in \Lambda_n}
a_\lambda \overline{a_{\lambda'}} \int_0^L \e^{i2\pi\langle \lambda -\lambda', \gamma(t)\rangle}\,dt - L.
\end{split}
\end{equation}
Now we are left to simplify the second summand in the right-hand side of \paref{proj2}. We get
\begin{equation}\label{2sum}
\begin{split}
\int_0^L H_2(\widetilde f_n'(t))\,dt &= \int_0^L \left (\left (\widetilde f_n'(t)\right )^2 -1 \right )\,dt=\frac{1}{2\pi^2n}\int_0^L \langle\nabla T_n(\gamma(t)), \dot \gamma(t)\rangle ^2\,dt - L\cr
&=\frac{1}{2\pi^2n}\int_0^L \frac{4\pi^2}{\mathcal N_n}\sum_{\lambda,\lambda'\in \Lambda_n} a_\lambda \overline{a_{\lambda'}}
\langle \lambda, \dot \gamma(t)\rangle\langle \lambda', \dot \gamma(t)\rangle\e^{i2\pi\langle \lambda -\lambda', \gamma(t)\rangle}\,dt - L\cr
&=2\frac{1}{\mathcal N_n}\sum_{\lambda,\lambda'\in \Lambda_n} a_\lambda \overline{a_{\lambda'}}\int_0^L
\left \langle \frac{\lambda}{|\lambda|}, \dot \gamma(t)\right \rangle \left \langle \frac{\lambda'}{|\lambda|}, \dot \gamma(t)\right \rangle\e^{i2\pi\langle \lambda -\lambda', \gamma(t)\rangle}\,dt - L.
\end{split}
\end{equation}
Using \paref{1sum} and \paref{2sum} in \paref{proj2} we obtain, taking into account that
$$b_2a_0=-1/(2\pi)=-b_0a_2,$$
\begin{equation}\label{eq1c}
\begin{split}
\mathcal Z_n[2]&=\sqrt{2\pi^2n} \left(  b_2a_0 \int_0^L H_2(f_n(t))\,dt
 + b_0 a_2 \int_0^L H_2(\widetilde f_n'(t))\,dt\right)\cr
&=\frac{\sqrt{2\pi^2n}}{2\pi} \left( - \frac{1}{\mathcal N_n} \sum_{\lambda, \lambda'\in \Lambda_n}
a_\lambda \overline{a_{\lambda'}} \int_0^L \e^{i2\pi\langle \lambda -\lambda', \gamma(t)\rangle}\,dt + L \right ) \cr
&+\frac{\sqrt{2\pi^2n}}{2\pi} \left(2\frac{1}{\mathcal N_n}\sum_{\lambda,\lambda'\in \Lambda_n} a_\lambda \overline{a_{\lambda'}}\int_0^L
\left \langle \frac{\lambda}{|\lambda|}, \dot \gamma(t)\right \rangle \left \langle \frac{\lambda'}{|\lambda|}, \dot \gamma(t)\right \rangle\e^{i2\pi\langle \lambda -\lambda', \gamma(t)\rangle}\,dt - L
\right )\cr
&=\frac{\sqrt{2\pi^2n}}{2\pi} \left( - L\frac{1}{\mathcal N_n} \sum_{\lambda\in \Lambda_n}
|a_\lambda|^2  + L +2\frac{1}{\mathcal N_n}\sum_{\lambda\in \Lambda_n} |a_\lambda|^2 \int_0^L
\left \langle \frac{\lambda}{|\lambda|}, \dot \gamma(t)\right \rangle^2\,dt - L
\right ) \cr
&+\frac{\sqrt{2\pi^2n}}{2\pi} \frac{1}{\mathcal N_n} \sum_{\lambda\ne \lambda'}
a_\lambda \overline{a_{\lambda'}} \int_0^L \left(2\left \langle \frac{\lambda}{|\lambda|}, \dot \gamma(t)\right \rangle \left \langle \frac{\lambda'}{|\lambda|}, \dot \gamma(t)\right \rangle   -1 \right)\e^{i2\pi\langle \lambda -\lambda', \gamma(t)\rangle}\,dt.
\end{split}
\end{equation}
Now, thanks to \new{Lemma \ref{lem:ident sum squares prod}}, we can write \paref{eq1c} as
\begin{equation*}
\begin{split}
\mathcal Z_n[2] &= \frac{\sqrt{2\pi^2n}}{2\pi}  \frac{1}{\mathcal N_n} \sum_{\lambda\in \Lambda_n}
(|a_\lambda|^2 -1) \left ( 2 \int_0^L
\left \langle \frac{\lambda}{|\lambda|}, \dot \gamma(t)\right \rangle^2\,dt  -L \right)\cr
&+\frac{\sqrt{2\pi^2n}}{2\pi} \frac{1}{\mathcal N_n} \sum_{\lambda\ne \lambda'}
a_\lambda \overline{a_{\lambda'}} \int_0^L \left(2\left \langle \frac{\lambda}{|\lambda|}, \dot \gamma(t)\right \rangle \left \langle \frac{\lambda'}{|\lambda|}, \dot \gamma(t)\right \rangle   -1 \right)\e^{i2\pi\langle \lambda -\lambda', \gamma(t)\rangle}\,dt,
\end{split}
\end{equation*}
which equals \paref{proj2formula}.

\end{proof}

\section{Auxiliary results for the approximate Kac-Rice formula}

\subsection{Contribution of singular squares}\label{Ssing}

In this section we prove Lemma \ref{lemmasing}, following \cite[\S 4]{R-W}. We first need the following result, whose proof is similar to the proof of Lemma 2.4 in \cite{R-W-Y} and hence omitted.
\begin{lemma}\label{lemmasmall}
There exists a constant $c_0>0$ sufficiently small such that for every $t_1,t_2\in [0,L]$ with
$$0<|t_1-t_2|< c_0/\sqrt{E_n},$$ we have
$$
r(t_1,t_2) \ne \pm 1.
$$
\end{lemma}
Proposition 4.4 in \cite{R-W} also asserts that
for $t_1\in [0,L]$ and $0<|t_2 - t_1| < c_0/\sqrt{n}$ one has the uniform estimate
\begin{equation}\label{igor_unif}
K_2(t_1,t_2) = O(n).
\end{equation}
Lemma \ref{lemmasmall} and \paref{igor_unif} allows to prove the following as in the proof of \cite[Proposition 4.1]{R-W}.
\begin{lemma}\label{Ovar}
We have
$$
\Var(\mathcal Z_i) = O(1),
$$
uniformly for  $i\le k$, where the constants involved in the ``O"-notation depend only on $c_0$.
\end{lemma}

The following follows upon applying Cauchy-\new{Schwarz} with Lemma \ref{Ovar}.
\begin{corollary}\label{Ocov}
We have
$$
\Cov(\mathcal Z_i, \mathcal Z_j) = O(1),
$$
uniformly for  $i,j\le k$, where the constants involved in the ``O"-notation depend only on $c_0$.
\end{corollary}
Let us now denote by $B$ the union of all singular cubes
\begin{equation}\label{defBsing}
B = \bigcup_{S_{i,j}\text{ singular}} S_{i,j}.
\end{equation}
The proof of the following is similar to the proof of Lemma 4.7 in \cite{R-W}.
\begin{lemma}\label{areaB}
The total area of the singular set is, for a $\delta$-separated sequence $\lbrace n\rbrace\subset S$ such that $\mathcal N_n\to +\infty$,
$$
\text{meas}(B) = o\left( \mathcal N_n^{-2}  \right).
$$
\end{lemma}
\begin{proof}
We apply Chebyshev-Markov inequality to the measure of $B$ to get
$$
\text{meas}(B) \ll \int_0^L r(t_1,t_2)^6\,dt_1 dt_2.
$$
Hence bounding the measure of the singular set $B$ is reduced to bounding the $6$th moment and its derivatives.
An application of Lemma \ref{lemma6} below then concludes the proof of Lemma \ref{areaB}.

\end{proof}

We are now in a position to prove Lemma \ref{lemmasing}.

\begin{proof}[Proof of Lemma \ref{lemmasing}]
Since the number of singular cubes is $O\left( E_n \text{meas}(B)\right)$, Corollary \ref{Ocov} bounds the contribution of singular cubes in \paref{varsum2} as
\begin{equation}\label{sing}
\left |\sum_{i,j : S_{i,j}\text{ sing.}} \Cov(\mathcal Z_i, \mathcal Z_j)\right | = O\left  ( E_n \cdot \text{meas}(B)\right );
\end{equation}
The statement of Lemma \ref{lemmasing} then follows upon an application of Lemma \ref{areaB}.

\end{proof}

\subsection{Taylor expansion for the two-point correlation function}\label{sectiontaylork2}
\begin{proof}[Proof of Lemma \ref{taylorK2}]
Recall that $\alpha=2\pi^2n$ as in \paref{alpha_var}. Lemma 3.2 in \cite{R-W} asserts that
\begin{equation}\label{K2igor}
K_2 = \frac{1}{\pi^2(1-r^2)^{3/2}}\cdot \mu \cdot (\sqrt{1-\rho^2} + \rho \arcsin \rho),
\end{equation}
where
$$
\mu = \sqrt{\alpha (1-r^2) - r_1^2}\cdot \sqrt{\alpha (1-r^2) - r_2^2},
$$

$$
\rho = \frac{r_{12}(1-r^2) + r r_1 r_2}{ \sqrt{\alpha (1-r^2) - r_1^2}\cdot \sqrt{\alpha (1-r^2) - r_2^2}}.
$$
We now set
$$
G(\rho) := \frac{2}{\pi} \left( \sqrt{1-\rho^2} + \rho \arcsin \rho   \right),
$$
then, as $\rho \to 0$,
\begin{equation}\label{TayG}
G(\rho)= \frac{2}{\pi} \left (1+ \frac{\rho^2}{2} +\frac{\rho^4}{24} + O(\rho^6) \right ).
\end{equation}
Let us now expand $\mu$ around $0$.
\begin{equation}\label{TayMu}
\begin{split}
\mu =
\alpha \Big(&1-r^2   - \frac{(r_2/\sqrt{\alpha})^2}{2}   - \frac{(r_1/\sqrt{\alpha})^2}{2}
 - \frac{(r_2/\sqrt{\alpha})^4}{8}  \cr
&- \frac{(r_1/\sqrt{\alpha})^4}{8}  +\frac{(r_1/\sqrt{\alpha})^2(r_2/\sqrt{\alpha})^2}{4}   +O((r^2+(r_2/\sqrt{\alpha})^2)^3) \Big).
\end{split}
\end{equation}
Moreover,
\begin{equation}\label{TayFrac}
\frac{1}{(1-r^2)^{3/2}}=1 + \frac32 r^2 + \frac{15}{8}r^4 + O(r^6).
\end{equation}
Let us now Taylor expand $\rho$.
\begin{equation}\label{TayRho}
\begin{split}
\rho
=&(r_{12}/\alpha)    + (r_{12}/\alpha) \frac{(r_2/\sqrt{\alpha})^2}{2}  + (r_{12}/\alpha) \frac{(r_1/\sqrt{\alpha})^2}{2} + r (r_1/\sqrt{\alpha})( r_2/\sqrt{\alpha})\cr &+O(r_{12}(r^2+(r_2/\sqrt{\alpha})^2)^2).
\end{split}
\end{equation}
Substituting \paref{TayRho} into \paref{TayG} we obtain
\begin{equation}\label{TayG2}
\begin{split}
G(\rho) = \frac{2}{\pi}\Big(& 1+\frac{1}{2} \left((r_{12}/\alpha)^2 +(r_{12}/\alpha)^2 (r_2/\sqrt{\alpha})^2  ) +(r_{12}/\alpha)^2 (r_1/\sqrt{\alpha})^2  )  \right) \cr
&+2(r_{12}/\alpha)r (r_1/\sqrt{\alpha})( r_2/\sqrt{\alpha})+\frac{1}{24}(r_{12}/\alpha)^4   +O(r_{12}^2 (r^2+(r_2/\sqrt{\alpha})^2)^2) \Big).
\end{split}
\end{equation}
Finally, using \paref{TayFrac}, \paref{TayMu} and \paref{TayG2} in \paref{K2igor} we get
\begin{equation*}
\begin{split}
K
=\frac{\alpha}{\pi^2}\Big(& 1+\frac{1}{2} (r_{12}/\alpha)^2 +\frac12 r^2   - \frac{(r_2/\sqrt{\alpha})^2}{2}   - \frac{(r_1/\sqrt{\alpha})^2}{2} \cr
&+\frac38 r^4+\frac{1}{24}(r_{12}/\alpha)^4- \frac{(r_2/\sqrt{\alpha})^4}{8}  - \frac{(r_1/\sqrt{\alpha})^4}{8} \cr
&+(r_{12}/\alpha)r (r_1/\sqrt{\alpha})( r_2/\sqrt{\alpha})+\frac{(r_1/\sqrt{\alpha})^2(r_2/\sqrt{\alpha})^2}{4}\cr
& - \frac{3}{2}r^2\frac{(r_2/\sqrt{\alpha})^2}{2}   - \frac{3}{2}r^2\frac{(r_1/\sqrt{\alpha})^2}{2}+\frac{1}{2} (r_{12}/\alpha)^2\frac12 r^2\cr
&+\frac14 (r_2/\sqrt \alpha)^2 (r_{12}/\alpha)^2 +\frac14 (r_1/\sqrt \alpha)^2 (r_{12}/\alpha)^2\cr
&+O(r^6 +(r_1/\sqrt{\alpha})^6 +(r_2/\sqrt{\alpha})^6 +(r_{12}/\alpha)^6  )\Big),
\end{split}
\end{equation*}
which is \paref{eqTayK2}.

\end{proof}

\section{Moments of $r$ and its derivatives}\label{Smoments}

\begin{lemma}\label{lemma2}
If $\mathcal C\subset \T$ is a smooth curve with nowhere vanishing curvature, then for a $\delta$-separated sequence $\lbrace n\rbrace$ such that $\mathcal N_n\to +\infty$, we have
\begin{equation*}
\begin{split}
&1) \int_0^L \int_0^L r(t_1, t_2)^2\,dt_1 dt_2 = \frac{L^2}{\mathcal N_n} + o\left(\frac{1}{\mathcal N_n^2}\right),\cr
&2) \int_0^L \int_0^L \left| \frac{1}{\sqrt{4\pi^2 n}} r_1(t_1, t_2)\right |^2\,dt_1 dt_2 = \frac{L^2}{2\mathcal N_n} + o\left(\frac{1}{\mathcal N_n^2}\right),\cr
&3) \int_0^L \int_0^L \left | \frac{1}{4\pi^2 n} r_{12}(t_1, t_2)\right |^2\,dt_1 dt_2 = \frac{B_{\mathcal C}(\mu_n)}{\mathcal N_n} + o\left(\frac{1}{\mathcal N_n^2}\right),
\end{split}
\end{equation*}
where $B_{\mathcal C}(\mu_n)$ is given in \paref{B}.
\end{lemma}
\begin{proof}
Let us start with $1)$. Squaring out, we have (on separating the diagonal $\lambda = \lambda'$)
\begin{equation}\label{app1}
\begin{split}
\int_0^L \int_0^L r(t_1, t_2)^2\,dt_1 dt_2 = \frac{L^2}{\mathcal N_n} + \frac{1}{\mathcal N_n^2} \sum_{\lambda\ne \lambda'} \left | \int_0^L \e^{i2\pi \langle \lambda -\lambda', \gamma(t)\rangle}\,dt\right|^2.
\end{split}
\end{equation}
Lemma 5.2 in \cite{R-W} yields
$$
\int_0^L \e^{i2\pi \langle \lambda -\lambda', \gamma(t)\rangle}\,dt\ll \frac{1}{|\lambda -\lambda'|^{1/2}},
$$
therefore the contribution of the off-diagonal pairs in \paref{app1} is
$$
\frac{1}{\mathcal N_n^2} \sum_{\lambda\ne \lambda'} \left | \int_0^L \e^{i2\pi \langle \lambda -\lambda', \gamma(t)\rangle}\,dt\right|^2\ll \frac{1}{\mathcal N_n^2} \sum_{\lambda\ne \lambda'} \frac{1}{|\lambda -\lambda'|}.
$$
Condition \paref{generic} then allows to conclude part $1)$. The remaining terms can be dealt in a similar way to the proof of
\cite[Proposition 5.1]{R-W}, taking into account \paref{generic} to control the contribution of the ``off-diagonal terms".

\end{proof}

\begin{proof}[Proof of Lemma \ref{lemma4}]
Let us prove $1)$.  We can write
\begin{equation}\label{long1}
\begin{split}
\int_{\mathcal C} \int_{\mathcal C} r(t_1,t_2)^4\,dt_1 dt_2 &=\int_{\mathcal C} \int_{\mathcal C}  \left(\frac{1}{\mathcal N_n}\sum_\lambda \e^{i2\pi\langle \lambda,\gamma(t_1)-\gamma(t_2)\rangle}    \right)^4\,dt_1 dt_2\cr
&=\int_{\mathcal C} \int_{\mathcal C} \frac{1}{\mathcal N_n^4}\sum_{\lambda_1,\dots,\lambda_4} \e^{i2\pi\langle \lambda_1-\lambda_2+\lambda_3-\lambda_4,\gamma(t_1)-\gamma(t_2)\rangle}    \,dt_1 dt_2
\end{split}
\end{equation}
Now let us split the summation on the r.h.s. of \paref{long1} into two sums: one over quadruples $(\lambda_1,\dots,\lambda_4)$ such that $\lambda_1+\dots + \lambda_4=0$ and the other one over quadruples $(\lambda_1,\dots,\lambda_4)$ such that $\lambda_1+\dots + \lambda_4\ne 0$:
\begin{equation}\label{long1new}
\begin{split}
\int_{\mathcal C} \int_{\mathcal C} r(t_1,t_2)^4\,dt_1 dt_2 &=L^2 \frac{|S_4(n)|}{\mathcal N_n^4}+ \int_{\mathcal C} \int_{\mathcal C} \frac{1}{\mathcal N_n^4}\sum_{\lambda_1-\lambda_2+\lambda_3-\lambda_4\ne 0} \e^{i2\pi\langle \lambda_1-\lambda_2+\lambda_3-\lambda_4,\gamma(t_1)-\gamma(t_2)\rangle}    \,dt_1 dt_2\cr
&=L^2 \frac{|S_4(n)|}{\mathcal N_n^4}+ \frac{1}{\mathcal N_n^4}\sum_{\lambda_1-\lambda_2+\lambda_3-\lambda_4\ne 0}  \left| \int_{\mathcal C}\e^{i2\pi\langle \lambda_1-\lambda_2+\lambda_3-\lambda_4,\gamma(t_1)\rangle}    \,dt_1\right|^2,
\end{split}
\end{equation}
where
$$
S_4(n) :=\lbrace (\lambda_1, \dots, \lambda_4)\in \Lambda_n^4 : \lambda_1 + \dots + \lambda_4 =0\rbrace.
$$
Recall that \cite{K-K-W}
\begin{equation}\label{cardin_4length}
|S_4(n)| = 3\mathcal N_n(\mathcal N_n -1),
\end{equation}
and moreover \cite[Lemma 5.2]{R-W}
\begin{equation}\label{estim4}
 \frac{1}{\mathcal N_n^4}\sum_{\lambda_1-\lambda_2+\lambda_3-\lambda_4\ne 0}  \left| \int_0^L\e^{i2\pi\langle \lambda_1-\lambda_2+\lambda_3-\lambda_4,\gamma(t_1)\rangle}    \,dt_1\right|^2\ll \frac{1}{\mathcal N_n^4}\sum_{\lambda_1-\lambda_2+\lambda_3-\lambda_4\ne 0} \frac{1}{|\lambda_1+\dots+\lambda_4|}.
\end{equation}
Substituting \paref{cardin_4length} and \paref{estim4} into \paref{long1new} we get
\begin{equation}\label{long2}
\begin{split}
\int_{\mathcal C} \int_{\mathcal C} r(t_1,t_2)^4\,dt_1 dt_2
=3L^2 \frac{1}{\mathcal N_n^2}+ O\left(  \frac{1}{\mathcal N_n^4}\sum_{\lambda_1-\lambda_2+\lambda_3-\lambda_4\ne 0} \frac{1}{|\lambda_1+\dots+\lambda_4|} \right ).
\end{split}
\end{equation}
Lemma \ref{remainder4}, in particular \paref{quantity4}, then allows to estimate the error on the r.h.s. of \paref{long2}, thus concluding the proof of $1)$.

Let us now deal with 4).
\begin{equation}
\begin{split}
&\int_{\mathcal C} \int_{\mathcal C} (r_{12}/\alpha)r (r_1/\sqrt{\alpha})( r_2/\sqrt{\alpha})\,dt_1 dt_2 \cr
&=4\cdot \int_{\mathcal C} \int_{\mathcal C} \frac{1}{\mathcal N_n^4}\sum_{\lambda_1,\dots,\lambda_4} \left \langle \frac{\lambda_1}{|\lambda_1|},\dot \gamma(t_1)\right \rangle \left \langle \frac{\lambda_1}{|\lambda_1|},\dot \gamma(t_2)\right \rangle \left \langle \frac{\lambda_3}{|\lambda_3|},\dot \gamma(t_1)\right \rangle \left
 \langle \frac{\lambda_4}{|\lambda_4|},\dot \gamma(t_2)\right \rangle\times \cr
&\times \e^{i2\pi \left \langle \lambda_1+\lambda_2+\lambda_3+\lambda_4,\gamma(t_1)-\gamma(t_2)\right \rangle}    \,dt_1 dt_2\cr
&=-4\cdot \int_{\mathcal C} \int_{\mathcal C} \frac{1}{\mathcal N_n^4}\sum_{\lambda,\lambda'} \left \langle \frac{\lambda}{|\lambda|},\dot \gamma(t_1)\right \rangle \left \langle \frac{\lambda}{|\lambda|},\dot \gamma(t_2)\right \rangle
\left \langle \frac{\lambda'}{|\lambda'|},\dot \gamma(t_1)\right \rangle \left \langle \frac{\lambda'}{|\lambda'|},\dot \gamma(t_2)\right \rangle \,dt_1 dt_2\cr
&-8\cdot \int_{\mathcal C} \int_{\mathcal C} \frac{1}{\mathcal N_n^4}\underbrace{\sum_{\lambda,\lambda'} \left \langle \frac{\lambda}{|\lambda|},\dot \gamma(t_1)\right \rangle^2 \left \langle \frac{\lambda}{|\lambda|},\dot \gamma(t_2)\right \rangle
\left \langle \frac{\lambda'}{|\lambda'|},\dot \gamma(t_2)\right \rangle}_{=0} \,dt_1 dt_2\cr
&+4\cdot \int_{\mathcal C} \int_{\mathcal C} \frac{1}{\mathcal N_n^4}\sum_{\lambda_1+\dots+\lambda_4\ne 0} \left \langle \frac{\lambda_1}{|\lambda_1|},\dot \gamma(t_1)\right \rangle \left \langle \frac{\lambda_1}{|\lambda_1|},\dot \gamma(t_2)\right \rangle \left \langle \frac{\lambda_3}{|\lambda_3|},\dot \gamma(t_1)\right \rangle \left \langle \frac{\lambda_4}{|\lambda_4|},\dot \gamma(t_2)\right \rangle\times \cr
&\times \e^{i2\pi\langle \lambda_1+\lambda_2+\lambda_3+\lambda_4,\gamma(t_1)-\gamma(t_2)\rangle}    \,dt_1 dt_2\cr
&=-4F_{\mathcal C}(\mu_n) \frac{1}{\mathcal N_n^2} + o(\mathcal N_n^{-2}),
\end{split}
\end{equation}
where in the last step we used \paref{defF}, and Cauchy-\new{Schwarz} inequality and again \paref{quantity4_estim} to bound the contribution of ``off-diagonal" terms.
Let us now study $5)$.
\begin{equation}
\begin{split}
 &\int_{\mathcal C} \int_{\mathcal C}(r_1/\sqrt{\alpha})^2(r_2/\sqrt{\alpha})^2\,dt_1 dt_2
\cr
&=4\cdot \int_{\mathcal C} \int_{\mathcal C} \frac{1}{\mathcal N_n^4}\sum_{\lambda_1,\dots,\lambda_4} \left \langle \frac{\lambda_1}{|\lambda_1|},\dot \gamma(t_1)\right \rangle \left \langle \frac{\lambda_2}{|\lambda_2|},\dot \gamma(t_1)\right \rangle \left \langle \frac{\lambda_3}{|\lambda_3|},\dot \gamma(t_2)\right \rangle \left \langle \frac{\lambda_4}{|\lambda_4|},\dot \gamma(t_2)\right \rangle\times \cr
&\times \e^{i2\pi\left \langle \lambda_1-\lambda_2+\lambda_3-\lambda_4,\gamma(t_1)-\gamma(t_2)\right \rangle}    \,dt_1 dt_2\cr
&=4\cdot \int_{\mathcal C} \int_{\mathcal C} \frac{1}{\mathcal N_n^4}\sum_{\lambda,\lambda'} \left \langle \frac{\lambda}{|\lambda|},\dot \gamma(t_1)\right \rangle^2 \left \langle \frac{\lambda'}{|\lambda'|},\dot \gamma(t_2)\right \rangle^2   \,dt_1 dt_2\cr
&+4\cdot 2\int_{\mathcal C} \int_{\mathcal C} \frac{1}{\mathcal N_n^4}\sum_{\lambda,\lambda'} \left \langle \frac{\lambda}{|\lambda|},\dot \gamma(t_1)\right \rangle \left \langle \frac{\lambda'}{|\lambda'|},\dot \gamma(t_1)\right \rangle \left \langle \frac{\lambda}{|\lambda|},\dot \gamma(t_2)\right \rangle \left \langle \frac{\lambda'}{|\lambda'|},\dot \gamma(t_2)\right \rangle   \,dt_1 dt_2\cr
&+ 4\cdot \int_{\mathcal C} \int_{\mathcal C} \frac{1}{\mathcal N_n^4}\sum_{\lambda_1+\dots+\lambda_4\ne 0} \left \langle \frac{\lambda_1}{|\lambda_1|},\dot \gamma(t_1)\right \rangle \left \langle \frac{\lambda_2}{|\lambda_2|},\dot \gamma(t_1)\right \rangle \left \langle \frac{\lambda_3}{|\lambda_3|},\dot \gamma(t_2)\right \rangle \left \langle \frac{\lambda_4}{|\lambda_4|},\dot \gamma(t_2)\right \rangle\times \cr
&\times \e^{i2\pi\langle \lambda_1-\lambda_2+\lambda_3-\lambda_4,\gamma(t_1)-\gamma(t_2)\rangle}    \,dt_1 dt_2\cr
&= L^2\frac{1}{\mathcal N_n^2}+4\cdot 2\frac{1}{\mathcal N_n^2}F_n+ o(\mathcal N_n^{-2}),
\end{split}
\end{equation}
where for the last step we used the well-known equality
$
\frac{1}{\mathcal N_n}\sum_{\lambda} \left \langle \frac{\lambda}{|\lambda|},v\right \rangle^2= \frac12
$
which holds for every unit vector $v$,
and still Cauchy-\new{Schwarz} inequality and then \paref{quantity4_estim} to bound the contribution of ``off-diagonal" terms.

The proof of 2) is analogous to that of 1), whereas the proofs of 3), 6)-8) are analogous to that of 5) above, and hence omitted.
\end{proof}

\begin{lemma}\label{lemma6}
If $\mathcal C\subset \T$ is a smooth curve with nowhere vanishing curvature, then for $\delta$-separated sequences $\lbrace n\rbrace$ such that $\mathcal N_n\to +\infty$, we have
\begin{equation}
\begin{split}
&1) \int_{\mathcal C} \int_{\mathcal C} r(t_1,t_2)^6\,dt_1 dt_2 = o(\mathcal N_n^{-2}),\cr
&2) \int_{\mathcal C} \int_{\mathcal C} (r_1/\sqrt \alpha)^6\,dt_1 dt_2 = o(\mathcal N_n^{-2}),\cr
&3) \int_{\mathcal C} \int_{\mathcal C} (r_{12}/\alpha)^6\,dt_1 dt_2 = o(\mathcal N_n^{-2}).
\end{split}
\end{equation}
\end{lemma}
\begin{proof}
Let us prove $1)$.
\begin{equation}\label{lunga1}
\begin{split}
\int_{\mathcal C} \int_{\mathcal C} r(t_1,t_2)^6\,dt_1 dt_2&=\int_{\mathcal C} \int_{\mathcal C} \left(\frac{1}{\mathcal N_n} \sum_{\lambda\in \Lambda_n} \e_\lambda(\gamma(t_1) - \gamma(t_2))\right)^6\,dt_1 dt_2\cr
&=\int_{\mathcal C} \int_{\mathcal C} \frac{1}{\mathcal N_n^6} \sum_{\lambda_1,\dots, \lambda_6\in \Lambda_n} \e_{\lambda_1+\dots + \lambda_6}(\gamma(t_1) - \gamma(t_2))\,dt_1 dt_2\cr
&=\frac{1}{\mathcal N_n^6}\sum_{\lambda_1,\dots, \lambda_6\in \Lambda_n}\left| \int_{\mathcal C}   \e_{\lambda_1+\dots + \lambda_6}(\gamma(t))\,dt \right|^2\cr
&= L^2 \frac{|S_6(n)|}{\mathcal N_n^6} + \frac{1}{\mathcal N_n^6} \sum_{\lambda_1+\dots + \lambda_6\ne 0}\left| \int_{\mathcal C}  \e_{\lambda_1+\dots + \lambda_6}(\gamma(t))\,dt \right|^2\cr
&\ll   \frac{|S_6(n)|}{\mathcal N_n^6} + \frac{1}{\mathcal N_n^6} \sum_{\lambda_1+\dots + \lambda_6\ne 0} \frac{1}{|\lambda_1 + \dots + \lambda_6|},
\end{split}
\end{equation}
where
$$
S_6(n) := \lbrace (\lambda_1, \dots, \lambda_6)\in \Lambda_n^6 : \lambda_1 + \dots + \lambda_6 =0\rbrace,
$$
and for the last step we used Lemma 5.2 in \cite{R-W}.
Recall now that \cite{B-B}
$$
|S_6(n)| = O\left( \mathcal N_n^{7/2} \right).
$$
Using the latter together with Lemma \ref{remainder6} (just below) in \paref{lunga1}, we can conclude the proof of $1)$. The proofs of the remaining cases $2), 3)$ are similar to that of $1)$ and hence omitted.

\end{proof}

\section{Contribution of ``off-diagonal" terms}\label{Soff-diagonal}

\subsection{Proof of Lemma \ref{remainder4}}

Let us start with the following simple lemma.

\begin{lemma}
\label{lem:angle chords}
Let $A,B,C,D$ be four points on $\mathbb S^{1}$ such that the segment $AC$ intersects the segment $BD$, and $O$ be the centre
of the circle. Then the angle
between $AC$ and $BD$ equals $$\frac{\angle AOB+\angle COD}{2}.$$
\end{lemma}

\begin{proof}

Let $E$ be the intersection point of $AC$ with $BD$ and $\alpha$ the corresponding angle. Then, as $\alpha$
is the external angle in the triangle $AED$ we have that
\begin{equation}
\alpha = \angle ADE + \angle EAD = \frac{\angle AOB + \angle COD}{2},
\end{equation}
since for any chord $XY$ on the circle, the angle $\angle XOY$ is twice
the angle $\angle XZY$ subtended by another point $Z$ on the circle.

\end{proof}

We are now in a position to prove Lemma \ref{remainder4}.

\begin{proof}[Proof of Lemma \ref{remainder4}]
First suppose that the segment $\lambda_{1}\lambda_{3}$ intersects the segment $\lambda_{2}\lambda_{4}$. For
$u = \lambda_{3}-\lambda_{1}$ and $w=\lambda_{4}-\lambda_{2}$, $v:=u-w$, and
$\alpha$ the angle between $u$ and $w$ we have
\begin{equation*}
\begin{split}
\|v\|^{2} &= \|u\|^{2}+\|w\|^{2} - 2\|u\|\cdot \|w\|\cos{\alpha} =(\|u\|-\|w\|)^{2} +
2\|u\|\cdot \|w\|(1-\cos{\alpha}) \\&\ge 2\|u\|\|w\|(1-\cos{\alpha}) \gg 2\|u\|\cdot \|w\| \cdot \alpha^{2},
\end{split}
\end{equation*}
provided that $\alpha \in [0,\pi]$ (say). Since $\|u\|\cdot \|w\| \gg n^{1/2+2\delta}$ by
assumption \eqref{generic}, and $$\alpha \gg n^{- 1/4+\delta}$$ by Lemma \ref{lem:angle chords},
we have
\begin{equation*}
\|v\|\gg n^{1+2\delta} \cdot n^{-1/2+2\delta} = n^{4\delta},
\end{equation*}
which, in case $\lambda_{1}\lambda_{3}$ intersects $\lambda_{2}\lambda_{4}$, is stronger than claimed.

Otherwise, let us assume that $\lambda_{1}\lambda_{3}$ does not intersect $\lambda_{2}\lambda_{4}$.
That means that both
$\lambda_{1},\lambda_{3}$ are lying on the same arc $\arc{\lambda_{2}\lambda_{4}}$ (one of two choices) and
either the arcs $\arc{\lambda_{2}\lambda_{1}}$, $\arc{\lambda_{1}\lambda_{3}}$,
$\arc{\lambda_{3}\lambda_{4}}$ are pairwise disjoint or
the arcs $\arc{\lambda_{4}\lambda_{1}}$, $\arc{\lambda_{1}\lambda_{3}}$ and $\arc{\lambda_{3}\lambda_{2}}$
are pairwise disjoint;
we assume w.l.o.g that the former holds. Since $$\lambda_{3}-\lambda_{1}=(-\lambda_{1})-(-\lambda_{3}),$$ and
upon replacing $(\lambda_{1},\lambda_{3})$ by $(-\lambda_{3},-\lambda_{1})$ if
necessary, we may assume that
$\lambda_{1},\lambda_{3}$ are lying in the smaller of the arcs
$\arc{\lambda_{2}\lambda_{4}}$, i.e. the angle $\lambda_{2}O\lambda_{4} < \pi.$

Denote as before $u = \lambda_{3}-\lambda_{1}$ and $w=\lambda_{4}-\lambda_{2}$, $v:=u-w$. Here we claim that
$|\|u\|-\|w\||$ is not too small, and then, by the triangle inequality, so is $$\|v\|\ge  |\|u\|-\|w\||.$$
Let $\alpha$ be the angle $\alpha=\angle \lambda_{2}O\lambda_{4}$, and the angle
$\beta = \lambda_{1}O\lambda_{3} < \alpha$.
We have $\|u\| = 2\sqrt n\sin{(\alpha/2)}$, $\|w\| = 2\sqrt n\sin{(\beta/2)}$, so that
\begin{equation*}
\|u\|-\|w\| = 2\sqrt n\sin\left (\frac{\alpha-\beta}{4} \right )\cos \left(\frac{\alpha+\beta}{4} \right).
\end{equation*}
However, by the assumption \eqref{generic} we have that both $\alpha-\beta \gg n^{-1/4+\delta}$ and
$2\pi-(\alpha+\beta) \gg n^{-1/4+\delta}$ as at least one of $\alpha$ or $\beta$ falls short of $\pi$ by
at least $\gg n^{-1/4+\delta}$. Hence
\begin{equation*}
\|u\|-\|w\| \gg \sqrt n \cdot n^{-1/4 +\delta}\cdot n^{-1/4 +\delta}=n^{2\delta}
\end{equation*}
which yields the statement of Lemma \ref{remainder4} in this case.

\end{proof}

\subsection{\new{Contribution of the higher order chaoses}}

\begin{lemma}\label{remainder6}
For $\delta$-separated sequences $\lbrace n\rbrace\subset S$ we have the bound
\begin{equation}
\label{eq:1/N^4 sum 1/l1+...l6 =o(1)}
\frac{1}{\mathcal N_n^4} \sum_{\lambda_1+\dots+\lambda_6\ne 0} \frac{1}{\|\lambda_1 + \dots + \lambda_6\|} = o_{\mathcal N_{n}\rightarrow \infty}
\left (1 \right ).
\end{equation}
\end{lemma}
\begin{proof}
Let $\epsilon>0$ be positive number, and set $A:=\frac{1}{2}\mathcal \Nc_n^{2+\epsilon}$ (say). We distinguish between two cases: \\
i)  $\|\lambda_1 + \dots + \lambda_6 \|\ge A$, ii) $\| \lambda_1 + \dots + \lambda_6\| < A$.
The contribution of all summands with i) holding is $$\ll \frac{1}{A} \ll \frac{1}{\Nc_{n}^{\epsilon}},$$ hence
we are only to bound the contribution of summands satisfying ii).

Assume here that ii) indeed holds.
We claim that for fixed $\lambda_1,\dots, \lambda_4$ such that $\lambda_{1}+\ldots+\lambda_{4}\ne 0$,
there exist $O(1)$-choices for $\lambda_5,\lambda_6$. Indeed, if
we assume that both $\|\lambda_1+\dots+\lambda_4 + \lambda_5 + \lambda_6 \| < A$ and $\|\lambda_1+\dots+\lambda_4 + \lambda'_5 + \lambda'_6 \| < A,$ then, by the triangle inequality, we obtain
\begin{equation}\label{dis}
\|\lambda_5 + \lambda_6 - \lambda_5' - \lambda_6'\| < 2 A.
\end{equation}
By Lemma \ref{remainder4} and in light of \eqref{eq:N=O(n^o(1))}, \paref{dis} is valid only if
$\lambda_5 + \lambda_6 - \lambda_5' - \lambda_6'=0$.
This in turn implies that either $\lambda_5=-\lambda_6$ or $\lambda_5=\lambda_5'$ or $\lambda_5=\lambda_6'$).
The possibility $\lambda_5=-\lambda_6$ is not valid, since then
$\lambda_1+\dots+\lambda_6=\lambda_1+\dots+\lambda_4$,
and therefore given $\lambda_{1},\ldots \lambda_{4}$ there could be at most two choices for $\lambda_{5},\lambda_{6}$
such that ii) holds.

Let $B>0$ be a large (but fixed) parameter, and assume that ii) holds. We distinguish between further two cases:
a) $\|\lambda_1 + \dots + \lambda_6 \|\ge B$, or b) $\|\lambda_1 + \dots + \lambda_6 \|< B$.
By the above, the contribution of summands satisfying a) to the l.h.s. of \eqref{eq:1/N^4 sum 1/l1+...l6 =o(1)} is
$O\left(\frac{1}{B}\right)$.

To treat terms that satisfy b) we recall that \cite[Theorem 2.2]{K-K-W} shows that the number of $6$-tuples
$(\lambda_{1},\ldots,\lambda_{6}) \in \Lambda_{n}^{6}$ satisfying $\lambda_{1}+\ldots+\lambda_{6}=0$ is $o(\Nc_{n}^{4})$.
A slight modification\footnote{In Eq. 59, consider $y_1+y_2 \in A+v$ and in Eq. 60, consider $y_1+y_2\in A+v$ so that we have $\langle 1_A \star 1_A, 1_{A+v}\rangle$} of the proof of \cite[Theorem 2.2]{K-K-W} yields that the same holds for an arbitrary {\em fixed}
$v\in \Z^{2}$ in place of $0$, i.e. the number of tuples such $6$-tuples satisfying
\begin{equation*}
\lambda_{1}+\ldots+\lambda_{6}=v
\end{equation*}
is $o(\Nc_{n}^{4})$. Therefore the total contribution to \eqref{eq:1/N^4 sum 1/l1+...l6 =o(1)} of summands satisfying b) is
$o_{B}(1)$.
Consolidating the various bounds we encountered, we have that
\begin{equation*}
\frac{1}{\mathcal N_n^4} \sum_{\lambda_1+\dots+\lambda_6\ne 0} \frac{1}{\|\lambda_1 + \dots + \lambda_6\|} \ll \frac{1}{\Nc_{n}^{\epsilon}}+\frac{1}{B}+o_{B}(1),
\end{equation*}
which certainly implies \eqref{eq:1/N^4 sum 1/l1+...l6 =o(1)}.
\end{proof}

\section{Auxiliary computations for the fourth chaos}\label{appendix4}
\begin{proof}[Proof of Lemma \ref{H1}]
Let us start with $1)$. Recall the expression for the fourth Hermite polynomial: $H_4(t) = t^4 -6t^2 +3$.
\begin{equation}\label{1}
\begin{split}
1) \int_0^L H_{4}(f_n(t))\,dt &= \int_0^L \left (f_n(t)^4 -6f_n(t)^2 +3 \right)\,dt\cr
&= \frac{1}{\mathcal N_n^2}\sum_{\lambda, \lambda', \lambda'', \lambda'''}a_\lambda \overline{a_{\lambda'}}a_{\lambda''}\overline{a_{\lambda'''}}\int_0^L
\e^{i2\pi\langle \lambda -\lambda' + \lambda'' -\lambda''', \gamma(t)\rangle}\,dt  \cr
&-6 \frac{1}{\mathcal N_n}\sum_{\lambda, \lambda'}a_\lambda \overline{a_{\lambda'}}\int_0^L
\e^{i2\pi\langle \lambda -\lambda', \gamma(t)\rangle}\,dt +3L.
\end{split}
\end{equation}
Let us divide the first term on the r.h.s. of \paref{1} into two summation, one over the quadruples in $S_4(n)$ and the other one over the quadruples not belonging to $S_4(n)$.
\begin{equation}\label{2}
\begin{split}
\int_0^L H_{4}(f_n(t))\,dt &= 3L\frac{1}{\mathcal N_n^2} \sum_{\lambda, \lambda'}|a_\lambda|^2 |a_{\lambda'}|^2 - 3L \frac{1}{\mathcal N_n^2} \sum_{\lambda}|a_\lambda|^4 \cr
&+\frac{1}{\mathcal N_n^2}\sum_{\lambda, \lambda', \lambda'', \lambda'''\notin S_n(4)}a_\lambda \overline{a_{\lambda'}}a_{\lambda''}\overline{a_{\lambda'''}}\int_0^L
\e^{i2\pi\langle \lambda -\lambda' + \lambda'' -\lambda''', \gamma(t)\rangle}\,dt\cr
&-6 \frac{1}{\mathcal N_n}\sum_{\lambda, \lambda'}a_\lambda \overline{a_{\lambda'}}\int_0^L
\e^{i2\pi\langle \lambda -\lambda', \gamma(t)\rangle}\,dt +3L.
\end{split}
\end{equation}
Now let us divide the fourth term on the r.h.s. of \paref{2} into two sums, one for the diagonal terms and the other one for the off-diagonal terms.
That is,
\begin{equation}\label{3}
\begin{split}
\int_0^L H_{4}(f_n(t))\,dt  & = 3L\frac{1}{\mathcal N_n^2} \sum_{\lambda, \lambda'}|a_\lambda|^2 |a_{\lambda'}|^2  -6L \frac{1}{\mathcal N_n}\sum_{\lambda}|a_\lambda|^2  +3L \cr
&-6 \frac{1}{\mathcal N_n}\sum_{\lambda\ne \lambda'}a_\lambda \overline{a_{\lambda'}}\int_0^L
\e^{i2\pi\langle \lambda -\lambda', \gamma(t)\rangle}\,dt - 3L \frac{1}{\mathcal N_n^2} \sum_{\lambda}|a_\lambda|^4\cr
&+\frac{1}{\mathcal N_n^2}\sum_{\lambda, \lambda', \lambda'', \lambda'''\notin S_n(4)}a_\lambda \overline{a_{\lambda'}}a_{\lambda''}\overline{a_{\lambda'''}}\int_0^L
\e^{i2\pi\langle \lambda -\lambda' + \lambda'' -\lambda''', \gamma(t)\rangle}\,dt.
\end{split}
\end{equation}
Note that the first three terms on the r.h.s. of \paref{3} can be rewritten as
\begin{equation*}
\begin{split}
\int_0^L H_{4}(f_n(t))\,dt  &= 3\cdot 2\cdot L \frac{1}{\mathcal N_n}\left ( \frac{1}{\sqrt{\mathcal N_n/2}} \sum_{\lambda\in \Lambda_n^+}(|a_\lambda|^2-1)\right)^2\cr   &- 3L \frac{1}{\mathcal N_n^2} \sum_{\lambda}|a_\lambda|^4
-6 \frac{1}{\mathcal N_n}\sum_{\lambda\ne \lambda'}a_\lambda \overline{a_{\lambda'}}\int_0^L
\e^{i2\pi\langle \lambda -\lambda', \gamma(t)\rangle}\,dt \cr
&+\frac{1}{\mathcal N_n^2}\sum_{\lambda, \lambda', \lambda'', \lambda'''\notin S_n(4)}a_\lambda \overline{a_{\lambda'}}a_{\lambda''}\overline{a_{\lambda'''}}\int_0^L
\e^{i2\pi\langle \lambda -\lambda' + \lambda'' -\lambda''', \gamma(t)\rangle}\,dt\cr
&= 6\cdot L \frac{1}{\mathcal N_n} W_1(n)^2  - 3L \frac{1}{\mathcal N_n^2} \sum_{\lambda}|a_\lambda|^4-6 \frac{1}{\mathcal N_n}\sum_{\lambda\ne \lambda'}a_\lambda \overline{a_{\lambda'}}\int_0^L
\e^{i2\pi\langle \lambda -\lambda', \gamma(t)\rangle}\,dt \cr
&+\frac{1}{\mathcal N_n^2}\sum_{\lambda, \lambda', \lambda'', \lambda'''\notin S_n(4)}a_\lambda \overline{a_{\lambda'}}a_{\lambda''}\overline{a_{\lambda'''}}\int_0^L
\e^{i2\pi\langle \lambda -\lambda' + \lambda'' -\lambda''', \gamma(t)\rangle}\,dt\cr
&= X_n^a + X_n^b,
\end{split}
\end{equation*}
where $W_1(n)$ is defined as in \paref{W1}, $X_n^a$ is given in \paref{xyz} and
\begin{equation}\label{4}
\begin{split}
X_n^b :=  &   - 3L \frac{1}{\mathcal N_n^2} \sum_{\lambda}(|a_\lambda|^4-2)-6 \frac{1}{\mathcal N_n}\sum_{\lambda\ne \lambda'}a_\lambda \overline{a_{\lambda'}}\int_0^L
\e^{i2\pi\langle \lambda -\lambda', \gamma(t)\rangle}\,dt \cr
&+\frac{1}{\mathcal N_n^2}\sum_{\lambda, \lambda', \lambda'', \lambda'''\notin S_n(4)}a_\lambda \overline{a_{\lambda'}}a_{\lambda''}\overline{a_{\lambda'''}}\int_0^L
\e^{i2\pi\langle \lambda -\lambda' + \lambda'' -\lambda''', \gamma(t)\rangle}\,dt.
\end{split}
\end{equation}
Let us prove that, as $\mathcal N_n\to +\infty$,
\begin{equation}\label{small o}
\Var(X_n^b) = o\left( \frac{1}{\mathcal N_n^2} \right).
\end{equation}
For the first term on the r.h.s. of \paref{4} we have
\begin{equation}
\Var\left ( \frac{1}{\mathcal N_n^2} \sum_{\lambda}(|a_\lambda|^4-2) \right) = O(\mathcal N_n^{-3})
\end{equation}
whereas for the second one \cite[(5.18)]{R-W} and \paref{generic} give
\begin{equation}
\Var\left ( \frac{1}{\mathcal N_n}\sum_{\lambda\ne \lambda'}a_\lambda \overline{a_{\lambda'}}\int_0^L
\e^{i2\pi\langle \lambda -\lambda', \gamma(t)\rangle} \right ) = o\left(\frac{1}{\mathcal N_n^2}  \right).
\end{equation}
For the last term
\begin{equation}
\Var\left (\frac{1}{\mathcal N_n^2}\sum_{\lambda, \lambda', \lambda'', \lambda'''\notin S_n(4)}a_\lambda \overline{a_{\lambda'}}a_{\lambda''}\overline{a_{\lambda'''}}\int_0^L
\e^{i2\pi\langle \lambda -\lambda' + \lambda'' -\lambda''', \gamma(t)\rangle} \right ) = o\left(\frac{1}{\mathcal N_n^2}  \right).
\end{equation}
follows from \cite[(5.18)]{R-W} and Lemma \ref{remainder4}. This concludes the proof of $1)$.

Let us now study the second summand in \paref{proj4}. The argument given below is similar to the one above concerning
the first summand in \paref{proj4}; accordingly we will omit some technical details.
\begin{equation}\label{5}
\begin{split}
\int_0^L H_{4}(f'_n(t))\,dt &= \int_0^L \left (f'_n(t)^4 -6f'_n(t)^2 +3 \right)\,dt\cr
&= \frac{1}{(2\pi^2n)^2}\int_0^L \frac{(2\pi)^4}{\mathcal N_n^2}\sum_{\lambda,\lambda',\lambda'',\lambda'''} a_\lambda \overline{a_{\lambda'}}a_{\lambda''} \overline{a_{\lambda'''}} \times \cr
&
\times \langle \lambda, \dot \gamma(t)\rangle\langle \lambda', \dot \gamma(t)\rangle\langle \lambda'', \dot \gamma(t)\rangle\langle \lambda''', \dot \gamma(t)\rangle  \e^{i2\pi\langle \lambda -\lambda'+\lambda''-\lambda''', \gamma(t)\rangle}\,dt\cr
& -6 \frac{1}{2\pi^2n}\int_0^L \frac{(2\pi)^2}{\mathcal N_n}\sum_{\lambda,\lambda'} a_\lambda \overline{a_{\lambda'}}
\langle \lambda, \dot \gamma(t)\rangle\langle \lambda', \dot \gamma(t)\rangle \e^{i2\pi\langle \lambda -\lambda', \gamma(t)\rangle}\,dt+3L.
\end{split}
\end{equation}
Dealing with \paref{5} as in \paref{3} we obtain
\begin{equation*}
\begin{split}
\int_0^L H_{4}(f'_n(t))\,dt &=3
\int_0^L \left(  2\frac{1}{\mathcal N_n}\sum_{\lambda} (|a_\lambda|^2 -1)\left\langle \frac{\lambda}{|\lambda|}, \dot \gamma(t)\right\rangle^2\right)^2\,dt \cr
&  -3 \frac{1}{(2\pi^2n)^2} \frac{(2\pi)^4}{\mathcal N_n^2}\sum_{\lambda} |a_\lambda|^4
\int_0^L \langle \lambda, \dot \gamma(t)\rangle^4\,dt\cr
&+\frac{1}{(2\pi^2n)^2}\int_0^L \frac{(2\pi)^4}{\mathcal N_n^2}\sum_{\lambda,\lambda',\lambda'',\lambda'''\notin S_n(4)} a_\lambda \overline{a_{\lambda'}}a_{\lambda''} \overline{a_{\lambda'''}}\times \cr
&\times \langle \lambda, \dot \gamma(t)\rangle\langle \lambda', \dot \gamma(t)\rangle\langle \lambda'', \dot \gamma(t)\rangle\langle \lambda''', \dot \gamma(t)\rangle \e^{i2\pi\langle \lambda -\lambda'+\lambda''-\lambda''', \gamma(t)\rangle}\,dt\cr
&-6 \frac{1}{2\pi^2n}\int_0^L \frac{(2\pi)^2}{\mathcal N_n}\sum_{\lambda\ne\lambda'} a_\lambda \overline{a_{\lambda'}}
\langle \lambda, \dot \gamma(t)\rangle\langle \lambda', \dot \gamma(t)\rangle \e^{i2\pi\langle \lambda -\lambda', \gamma(t)\rangle}\,dt\cr
&= Y_n^a + Y_n^b,
\end{split}
\end{equation*}
where $Y_n^a$ is defined as in \paref{xyz} and
\begin{equation*}
\begin{split}
Y_n^b := & -3 \frac{1}{(2\pi^2n)^2} \frac{(2\pi)^4}{\mathcal N_n^2}\sum_{\lambda} (|a_\lambda|^4-2)
\int_0^L \langle \lambda, \dot \gamma(t)\rangle^4\,dt\cr
&+\frac{1}{(2\pi^2n)^2}\int_0^L \frac{(2\pi)^4}{\mathcal N_n^2}\sum_{\lambda,\lambda',\lambda'',\lambda'''\notin S_n(4)} a_\lambda \overline{a_{\lambda'}}a_{\lambda''} \overline{a_{\lambda'''}}\times \cr
&\times \langle \lambda, \dot \gamma(t)\rangle\langle \lambda', \dot \gamma(t)\rangle\langle \lambda'', \dot \gamma(t)\rangle\langle \lambda''', \dot \gamma(t)\rangle \e^{i2\pi\langle \lambda -\lambda'+\lambda''-\lambda''', \gamma(t)\rangle}\,dt\cr
&-6 \frac{1}{2\pi^2n}\int_0^L \frac{(2\pi)^2}{\mathcal N_n}\sum_{\lambda\ne\lambda'} a_\lambda \overline{a_{\lambda'}}
\langle \lambda, \dot \gamma(t)\rangle\langle \lambda', \dot \gamma(t)\rangle \e^{i2\pi\langle \lambda -\lambda', \gamma(t)\rangle}\,dt.
\end{split}
\end{equation*}
We are now left with the third summand in \paref{proj4}.
\begin{equation}\label{equazione1}
\begin{split}
\int_0^L H_{2}(f_n(t))H_{2}(f'_n(t))\,dt &= \int_0^L (f_n(t)^2-1)(f'_n(t)^2-1)\,dt\cr
&=\frac{(2\pi)^2}{2\pi^2n\mathcal N_n^2}\int_0^L \sum_{\lambda,\lambda',\lambda'',\lambda'''} a_\lambda \overline{a_{\lambda'}}a_{\lambda''} \overline{a_{\lambda'''}}\times\cr
&\times
\langle \lambda'', \dot \gamma(t)\rangle\langle \lambda''', \dot \gamma(t)\rangle \e^{i2\pi\langle \lambda -\lambda'+\lambda''-\lambda''', \gamma(t)\rangle}\,dt\cr
&- \frac{1}{\mathcal N_n}\sum_{\lambda, \lambda'}a_\lambda \overline{a_{\lambda'}}\int_0^L
\e^{i2\pi\langle \lambda -\lambda', \gamma(t)\rangle}\,dt\cr
&-\frac{1}{2\pi^2n}\int_0^L \frac{(2\pi)^2}{\mathcal N_n}\sum_{\lambda,\lambda'} a_\lambda \overline{a_{\lambda'}}
\langle \lambda, \dot \gamma(t)\rangle\langle \lambda', \dot \gamma(t)\rangle \e^{i2\pi\langle \lambda -\lambda', \gamma(t)\rangle}\,dt+L
\end{split}
\end{equation}
Finally, we can rewrite \paref{equazione1} as
\begin{equation*}
\begin{split}
&\int_0^L H_{2}(f_n(t))H_{2}(f'_n(t))\,dt \cr
&=\left( \frac{1}{\mathcal N_n} \sum_\lambda (|a_\lambda|^2 -1) \right )
\left( \frac{1}{\mathcal N_n} \sum_{\lambda'} (|a_\lambda'|^2 -1)  \int_0^L 2 \left \langle \frac{\lambda'}{|\lambda'|}, \dot \gamma(t)\right \rangle^2\,dt   \right )\cr
&-\frac{(2\pi)^2}{2\pi^2n\mathcal N_n^2}\sum_{\lambda} |a_\lambda|^4\int_0^L
\langle \lambda, \dot \gamma(t)\rangle^2\,dt\cr
&+\frac{(2\pi)^2}{2\pi^2n\mathcal N_n^2}\sum_{\lambda,\lambda',\lambda'',\lambda'''\notin S_n(4)} a_\lambda \overline{a_{\lambda'}}a_{\lambda''} \overline{a_{\lambda'''}}\int_0^L
\langle \lambda'', \dot \gamma(t)\rangle\langle \lambda''', \dot \gamma(t)\rangle\times \cr
&\times  \e^{i2\pi\langle \lambda -\lambda'+\lambda''-\lambda''', \gamma(t)\rangle}\,dt- \frac{1}{\mathcal N_n}\sum_{\lambda\ne \lambda'}a_\lambda \overline{a_{\lambda'}}\int_0^L
\e^{i2\pi\langle \lambda -\lambda', \gamma(t)\rangle}\,dt\cr
&-\frac{1}{2\pi^2n}\int_0^L \frac{(2\pi)^2}{\mathcal N_n}\sum_{\lambda\ne\lambda'} a_\lambda \overline{a_{\lambda'}}
\langle \lambda, \dot \gamma(t)\rangle\langle \lambda', \dot \gamma(t)\rangle \e^{i2\pi\langle \lambda -\lambda', \gamma(t)\rangle}\,dt\cr
&= Z_n^a + Z_n^b,
\end{split}
\end{equation*}
where $Z_n^a$ is defined as in \paref{xyz} and
\begin{equation*}
\begin{split}
Z_n^b := & -\frac{(2\pi)^2}{2\pi^2n\mathcal N_n^2}\sum_{\lambda} (|a_\lambda|^4-2)\int_0^L
\langle \lambda, \dot \gamma(t)\rangle^2\,dt\cr
&+\frac{(2\pi)^2}{2\pi^2n\mathcal N_n^2}\sum_{\lambda,\lambda',\lambda'',\lambda'''\notin S_n(4)} a_\lambda \overline{a_{\lambda'}}a_{\lambda''} \overline{a_{\lambda'''}}\int_0^L
\langle \lambda'', \dot \gamma(t)\rangle\langle \lambda''', \dot \gamma(t)\rangle \e^{i2\pi\langle \lambda -\lambda'+\lambda''-\lambda''', \gamma(t)\rangle}\,dt\cr
&- \frac{1}{\mathcal N_n}\sum_{\lambda\ne \lambda'}a_\lambda \overline{a_{\lambda'}}\int_0^L
\e^{i2\pi\langle \lambda -\lambda', \gamma(t)\rangle}\,dt\cr
&-\frac{1}{2\pi^2n}\int_0^L \frac{(2\pi)^2}{\mathcal N_n}\sum_{\lambda\ne\lambda'} a_\lambda \overline{a_{\lambda'}}
\langle \lambda, \dot \gamma(t)\rangle\langle \lambda', \dot \gamma(t)\rangle \e^{i2\pi\langle \lambda -\lambda', \gamma(t)\rangle}\,dt.
\end{split}
\end{equation*}
\end{proof}

\section{A family of static curves}\label{Sdmitri}

\begin{proposition}
\label{prop:Dima invar 2pi/k}
Let $\Cc\subset\mathbb T$ be a smooth closed curve with nowhere vanishing curvature, invariant w.r.t. rotation by $2\pi/k$ for some $k\ge 3$.
Then $\Cc$ is static.
\end{proposition}

\begin{proof}
To show that $\Cc$ is static it is sufficient \cite[Corollary 7.2]{R-W} to see that
\begin{equation}
\label{eq:B=L^2/4}
B_{\Cc}\left(\frac{d\theta}{2\pi}  \right) = \frac{L^{2}}{4},
\end{equation}
i.e. check the condition \eqref{eqstatic} with $\mu = \frac{d\theta}{2\pi}$ only.
By inverting the order of integration in \eqref{B} we have that \eqref{eq:B=L^2/4} is equivalent to
\begin{equation}
\label{eq:iint <t,gamma>^2=L^2/4}
\int\limits_{0}^{2\pi}\left( \int\limits_{0}^{L}\langle \theta, \dot{\gamma}(t)\rangle ^{2} dt\right)^{2} \frac{d\theta}{2\pi} = \frac{L^{2}}{4}.
\end{equation}
We claim that for every $\theta\in [0,2\pi]$ the integrand above is
\begin{equation}
\label{eq:int <t,gamma>^2=L/2}
\int\limits_{0}^{L}\langle \theta, \dot{\gamma}(t)\rangle ^{2}dt = \frac{L}{2},
\end{equation}
which is certainly sufficient for \eqref{eq:iint <t,gamma>^2=L^2/4}.

To prove it we denote $\dot{\gamma}(t)=:e^{i\phi(t)}$ so that the integrand on the l.h.s. of \eqref{eq:int <t,gamma>^2=L/2} is
\begin{equation*}
\langle \theta, \dot{\gamma}(t)\rangle ^{2} = \cos(\theta-\phi(t))^{2},
\end{equation*}
and by the assumed invariance of $\Cc$ w.r.t. rotations by $2\pi/k$ we have that
\begin{equation}
\label{eq:int0Lcos(t-phi)^2 sep}
\int\limits_{0}^{L}\cos(\theta-\phi(t))^{2}dt
= \int\limits_{0}^{L/k}
\left( \sum\limits_{j=0}^{k-1}\cos\left(\theta-\phi(t)+j\cdot \frac{2\pi}{k}\right)^{2}\right)dt.
\end{equation}
Now, since, by assumption, $k\ge 3$, we have that
\begin{equation*}
\begin{split}
\sum\limits_{j=0}^{k-1}\cos\left(\theta-\phi(t)+j\cdot \frac{2\pi}{k}\right)^{2}
=\frac{k}{2} + \sum\limits_{j=0}^{k-1} \cos\left(2\left(\theta-\phi(t) + \frac{2\pi}{k}\right)\right) = \frac{k}{2},
\end{split}
\end{equation*}
which we substitute into \eqref{eq:int0Lcos(t-phi)^2 sep} to obtain
\begin{equation*}
\int\limits_{0}^{L}\cos(\theta-\phi(t))^{2}dt = \frac{k}{2}\cdot \frac{L}{k} = \frac{L}{2},
\end{equation*}
that is, we obtain \eqref{eq:int <t,gamma>^2=L/2}, which, as it was mentioned above, is sufficient to yield the
statement of Proposition \ref{prop:Dima invar 2pi/k}.

\end{proof}

\bigskip

\noindent \mnew{\sc MAP5-UMR CNRS 8145, Universit\'e Paris Descartes, France}

\noindent \emph{E-mail address:} \mnew{\texttt{maurizia.rossi@parisdescartes.fr}}

\bigskip

\noindent {\sc Department of Mathematics, King's College London, United Kingdom}

\noindent \emph{E-mail address:} \texttt{igor.wigman@kcl.ac.uk}

\end{document}